\documentclass[11pt]{article}

\title{The completion numbers of Hamiltonicity and pancyclicity in random graphs}

\author{
Yahav Alon
\thanks{School of Mathematical Sciences, Raymond and Beverly Sackler Faculty of Exact Sciences, Tel Aviv University,
Tel Aviv, 6997801, Israel. Email: yahavalo@tauex.tau.ac.il.}
\and Michael Anastos
\thanks{Institute of Science and Technology Austria, Klosterneurburg 3400, Austria. Email: michael.anastos@ist.ac.at.  This project has received funding from the European Union's Horizon 2020 research and innovation
programme under the Marie Sk\l{}odowska-Curie grant agreement No 101034413
\includegraphics[width=5.5mm, height=4mm]{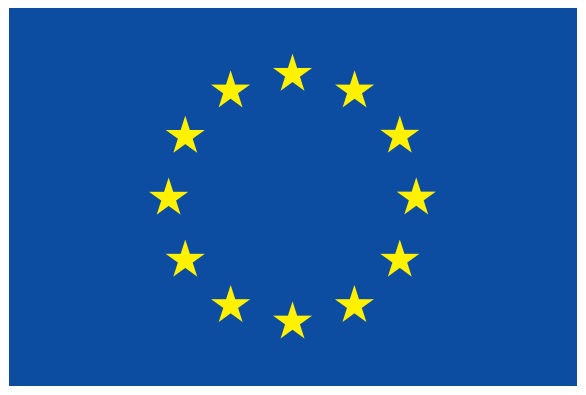}.}}

\usepackage[hidelinks]{hyperref}

\usepackage{tikz}
\usepackage{mathtools}
\usepackage{verbatim}
\usetikzlibrary{arrows,%
                petri,%
                topaths,%
                decorations.pathmorphing,%
                arrows.meta}%

\tikzset{snake it/.style={decorate, decoration=snake}}

\usepackage[position=top]{subfig}
\usepackage{amsmath,amsthm, amssymb,latexsym, amsfonts}
\usepackage{algorithm}
\usepackage{enumitem}

\begin{document}
\maketitle
\newtheorem{thm}{Theorem}%[section]
\newtheorem*{thm*}{Theorem}
\newtheorem{propos}{Proposition}
\newtheorem{defin}{Definition}
\newtheorem{lemma}{Lemma}[section]
\newtheorem{corol}{Corollary}[section]
\newtheorem{remark}{Remark}[section]

\newtheorem{corol*}{Corollary}
\newtheorem{thmtool}{Theorem}[section]
\newtheorem{corollary}[thmtool]{Corollary}
\newtheorem{lem}[thmtool]{Lemma}
\newtheorem{defi}[thmtool]{Definition}
\newtheorem{prop}[thmtool]{Proposition}
\newtheorem{clm}[thmtool]{Claim}
\newtheorem{conjecture}{Conjecture}
\newtheorem{problem}{Problem}
\newtheorem{observation}[thmtool]{Observation}
\newcommand{\Proof}{\noindent{\bf Proof.}\ \ }
\newcommand{\Remarks}{\noindent{\bf Remarks:}\ \ }
\newcommand{\Remark}{\noindent{\bf Remark:}\ \ }

\newcommand{\Dist}[1]{\mathsf{#1}}
\newcommand{\Bin}{\Dist{Bin}}
\newcommand{\cH}{\mathcal{H}}
\newcommand{\cL}{\mathcal{L}}
\newcommand{\cE}{\mathcal{E}} 
\newcommand{\cP}{\mathcal{P}}
\newcommand{\cT}{\mathcal{T}} 
\newcommand{\cc}{\gamma} 

\newcommand{\pr}{\mathbb{P}}
\newcommand{\sm}{\text{SMALL}}
\newcommand{\cl}{\text{CLOSE}}
\newcommand{\bd}{\text{BAD}}
\newcommand{\lrg}{\text{LARGE}}
\newcommand{\enm}{\text{END}}
\newcommand{\Lmax}{L_{\max}}

\newcommand{\YA}[1]{{\color{red} \small{(YA: #1)}}}
\newcommand{\MA}[1]{{\color{blue} \small{(MA: #1)}}}

\newcommand\bfrac[2]{\left(\frac{#1}{#2}\right)}

\begin{abstract}

Let $\mu(G)$ denote the minimum number of edges whose addition to $G$ results in a Hamiltonian graph, and let $\hat{\mu}(G)$ denote the minimum number of edges whose addition to $G$ results in a pancyclic graph. We study the distributions of $\mu (G),\hat{\mu}(G)$ in the context of binomial random graphs.

Letting $d=d(n) \coloneqq n\cdot p$, we prove that there exists a function $f:\mathbb{R}^+\to [0,1]$ of order $f(d) = \frac{1}{2}de^{-d}+e^{-d}+O(d^6e^{-3d})$ such that, if $G\sim G(n,p)$ with $20 \le d(n) \le 0.4 \log n$, then with high probability $\mu (G)= (1+o(1))\cdot f(d)\cdot n$.

Let $n_i(G)$ denote the number of degree $i$ vertices in $G$. A trivial lower bound on $\mu(G)$ is given by the expression $n_0(G) + \lceil \frac{1}{2}n_1(G) \rceil$. We show that in the random graph process $\{ G_t\}_{t=0}^{\binom{n}{2}}$ with high probability there exist times $t_1,t_2$, both of order $(1+o(1))n\cdot \left( \frac{1}{6}\log n + \log \log n \right)$, such that $\mu(G_t)=n_0(G_t) + \lceil \frac{1}{2}n_1(G_t) \rceil$ for every $t\ge t_1$ and $\mu(G_t)>n_0(G_t) + \lceil \frac{1}{2}n_1(G_t) \rceil$ for every $t\le t_2$. The time $t_i$ can be characterized as the smallest $t$ for which $G_t$ contains less than $i$ copies of a certain problematic subgraph. In particular, this implies that the hitting time for the existence of a Hamilton path is equal to the hitting time of $n_1(G) \le 2$ with high probability.
For the binomial random graph, this implies that if $np-\frac{1}{3}\log n - 2\log \log n \to \infty$ and $G\sim G(n,p)$ then, with high probability, $\mu (G) = n_0(G) + \lceil \frac{1}{2}n_1(G) \rceil$.

For completion to pancyclicity, we show that if $G\sim G(n,p)$ and $np\ge 20$ then, with high probability, $\hat{\mu} (G)=\mu (G)$.

Finally, we present a polynomial time algorithm such that, if $G\sim G(n,p)$ and $np\ge 20$, then, with high probability, the algorithm returns a set of edges of size $\mu (G)$ whose addition to $G$ results in a pancyclic (and therefore also Hamiltonian) graph.
\end{abstract}

\section{Introduction} \label{sec-intro}

Hamilton cycles are a central subject in the study of graphs, a fact which extends to the study of random graphs, with numerous and significant results regarding Hamilton cycles in random graphs obtained over the years.

We write $G\sim G(n,p)$ to denote that $G$ is distributed according to the binomial random graph model $G(n,p)$, i.e. $G$ is a random graph on $n$ vertices, where each edge of $G$ appears independently at random with probability $p$.  
A classical result by Koml\'{o}s and Szemer\'{e}di \cite{KS83}, and independently by Bollob\'{a}s \cite{B84}, states that if $G\sim G(n,p)$ then
\begin{eqnarray}\label{eq:ham}
\lim _{n\to \infty} \pr (G\text{ is Hamiltonian}) =
	\begin{cases}
        1 & \text{if } np-\log n -\log \log n \to \infty ;\\
        e^{-e^{-c}} & \text{if } np-\log n -\log \log n \to c ;\\
        0 & \text{if } np-\log n -\log \log n \to -\infty .
    \end{cases}
\end{eqnarray}

At and below the Hamiltonicity threshold, that is, when $np-\log n -\log \log n$ does not tend to plus infinity, one can examine the question of how far the random graph is from being Hamiltonian. In this paper, we quantify this distance by the notion of the \emph{completion number}.

\begin{defin}
Let $G$ be a graph and let $\mathcal{P}$ be a monotone increasing graph property. The \emph{completion number} of $G$ with respect to $\mathcal{P}$ is the minimum number of edges whose addition to $G$ results in a graph that satisfies $\mathcal{P}$.
\end{defin}

In other words, the completion number of $G$ with respect to $\cP$ is the edge distance between $G$ and a nearest graph that satisfies $\cP$. For example, the completion number of $G$ with respect to connectivity is equal to the number of connected components in $G$, minus 1. The completion number of $G$ with respect to containing a (near) perfect matching is equal to $\lfloor \frac{n}{2} \rfloor - \nu (G)$, where $\nu (G)$ is the maximum size of a matching in $G$. We let $\mu (G)$ denote the completion number of $G$ with respect to Hamiltonicity.

A disjoint path cover of $G$ is a set of vertex disjoint paths that cover all the vertices of $G$, here we use the convention that an isolated vertex corresponds to a path of length $0$. In addition we slightly abuse the notation and say that $G$ has a disjoint path cover of size $0$, whenever $G$ is Hamiltonian. Evidently, $\mu (G)$ is equal to the minimum size of a disjoint path cover of $G$. 

The result of Koml\'{o}s and Szemer\'{e}di, and of Bollob\'{a}s, stated above, implies that  if $G\sim G(n,p)$ and $np-\log n-\log\log n \to \infty$, then $\mu(G)=0$ with high probability\footnote{We say that a sequence of events $\{\cE_n\}_{n\geq 1}$ occurs with high probability if $\lim_{n\to \infty} \pr(\cE_n)=1$.}. Recently, Alon and Krivelevich \cite{AK} showed that, if $c$ is a large enough constant with respect to $\varepsilon$, then  $(1-\varepsilon )\frac{1}{2}ce^{-c}n \le \mu (G(n,c/n)) \le (1+\varepsilon )\frac{1}{2}ce^{-c}n$ with high probability.

Our first main result in this paper extends this last result in two ways, first by extending the range of edge probabilities $p$ for which this approximation holds, and second by replacing the error term $\varepsilon$ with an $o(1)$ error term, and the estimation $\frac{1}{2}ce^{-c}$ with a more accurate function $f(d)$ of $d \coloneqq np$.

Unfortunately, we are not able to derive a close form for the function $f(d)$. Instead, in the proof of  Theorem \ref{thm:main1}, we implicitly give an algorithm that calculates $f(d)$ within arbitrary accuracy by expressing it as a sum of terms of the form $c_{ij}d^ie^{-jd}$ with $\sum_{j\geq k}\sum_{i\geq 0}c_{ij}d^ie^{-jd} \leq 0.8^k$, for $k\geq 1$. The terms $c_{ij}d^ie^{-jd}$ will correspond to weighted sums of rooted subgraph counts. In Theorem \ref{thm:main1}, stated below, we give the first such few terms.

For $a\in \{d,n\}$ we write $g(a)=O_r(h(a))$ to mean that $|g(a)|\leq C_ah(a)$ for some constant $C_a$ that depends on $a$. We drop the subscript on $O(\cdot)$ whenever that is $n$.  

\begin{thm}\label{thm:main1}
Let $p=p(n)$ be such that\footnote{Here and going forward, all logarithms are assumed to be in the natural base.} $20 \le np \le 0.4  \log n$, and denote $d=d(n) \coloneqq np$. Let $G\sim G(n,p)$. Then, there exists a function $f:\mathbb{R}^+\to [0,1]$ of the order
$$
f(d) = \frac{1}{2}de^{-d} + e^{-d}+\left(\frac{1}{12}d^6 + \frac{1}{4}d^5 +\frac{1}{4}d^4 +\frac{1}{12}d^3 \right) \cdot e^{-3d} + O_d\left( d^{15}e^{-4d} \right)
$$
such that $\mu (G)= (1+o(1))\cdot f(d)\cdot n$ with high probability.
\end{thm}

Our second main result of this paper provides a characterisation of $\mu (G(n,p))$ in the denser regime.
Let $n_i(G)$ denote the number of vertices in $G$ with degree $i$. Observe that, for every graph $G$, the inequality $\mu (G) \ge n_0(G)+\lceil \frac{1}{2} n_1 (G) \rceil$ holds, since all the degrees in a Hamiltonian graph are at least 2. The two leading terms $\frac{1}{2}de^{-d}n + e^{-d}n$ in $f(d)\cdot n$ correspond to the expectation of the random variable $n_0(G(n,p))+\lceil \frac{1}{2} n_1 (G(n,p)) \rceil$. Our next result shows that, in fact, if $np-\frac{1}{3}\log n -2\log \log n \to \infty$ and $G\sim G(n,p)$ then, with high probability, $\mu (G) = n_0(G)+\lceil \frac{1}{2} n_1 (G) \rceil$. Note that this complements Theorem \ref{thm:main1} in the range $np\ge 0.4\log n$, as the distribution of $n_i(G)$ is well known.

In fact, we prove something stronger as we consider the evolution of 
$\{\mu(G_t)\}_{t=0}^{\binom{n}{2}}$ at the regime where $t$ is sufficiently large, where $\{ G_t \} _{t=0}^{\binom{n}{2}}$ denotes the random graph process on $[n]$, and the result on $G(n,p)$ will follow (See Corollary \ref{corol2}).

A $3$-spider of $G$ is a tree-subgraph of $G$ on $7$ vertices that is spanned by the edges incident to $3$ vertices of degree $2$ in $G$ with a common neighbour (see Figure \ref{fig:3-spider}).

\begin{figure}[h]
\centering
\begin{tikzpicture}[vertex/.style={draw,circle,color=black,fill=black,inner sep=1,minimum width=4pt},scale=1]
    
    \draw[thick] (0,0) circle (2);
    \node at (0,-0.6) {$V(G) \setminus \{u,v,w\}$};
	
	\node[vertex] (u) at (-1.35,1.9) {};
    \node[vertex] (v) at (-1.7,1.65) {};
    \node[vertex] (w) at (-2,1.35) {};
    
    \node[vertex] (x) at (-0.75,1.55) {};
\node[vertex] (common) at (-1.05,1.2) {};
    \node[vertex] (y) at (-1.35,0.85) {};
    \node[vertex] (z) at (-1.65,0.5) {};    
    
    \draw[thick] (u) to (x);
    \draw[thick] (v) to (y);
    \draw[thick] (w) to (z);
    
    \draw[thick] (u) to (common);
    \draw[thick] (v) to (common);
    \draw[thick] (w) to (common);    
    
	\node at (-1.35,2.2) {$u$};
	\node at (-1.8,1.95) {$v$};
	\node at (-2.2,1.65) {$w$};

    \node at (-0.55,1.35) {$a$};
    \node at (-0.85,1) {$b$};
    \node at (-1.15,0.65) {$c$};
    \node at (-1.45,0.3) {$d$};    
 
\end{tikzpicture}
\caption{\textit{A $3$-spider in $G$ with vertex set $\{a,b,c,d,u,v,w\}$. The vertices $u,v,w$ have degree $2$ and a common neighbour, and the edges incident to them span a tree}.}\label{fig:3-spider}
\end{figure}

It turns out that for enough large $t$, the difference $\mu (G_t) - \left( n_0(G_t) + \lceil \frac{1}{2} n_1 (G_t)\rceil \right)$ is strongly connected with the number of $3$-spiders in $G_t$. This number corresponds to the leading term of $f(d)-\frac{1}{2}de^{-d}-e^{-d}$ in the expansion of $f(d)$, which is $\frac{1}{12}\cdot d^6 \cdot e^{-3d}$. We also describe typical behaviour of $\mu (G_t)$ at the very sparse range $t=o(n)$. In this range, the difference $\mu (G_t) - \left( n_0(G_t) + \lceil \frac{1}{2} n_1 (G_t)\rceil \right)$ is determined by the multiplicity of trees with more than $2$ leaves. The smallest such tree, and hence the one that appears first in the random process, is the $3$-star $K_{1,3}$. This motivates the following definition.

\begin{defin}\label{def:times}
Let $\{ G_t \} _{t=0}^{\binom{n}{2}}$ be a random graph process. For $i\in \mathbb{N}^+$ we define the random variable $t_i$ as the minimum $t$ greater than $10n$ such that $G_t$ spans less than $i$ distinct $3$-spiders. We further define $t_i^*$ to be the minimum $t$ such that $G_t$ spans at least $i$ copies of $K_{1,3}$.
\end{defin}

Note that both the property of containing less than $i$ 3-spiders and the property of containing at least $i$ 3-stars are not monotone properties, so $t_i$ and $t_i^*$ should not be considered as hitting times.

Following Definition \ref{def:times} we are now ready to state our second main theorem.

\begin{thm}\label{main2}
Let $\{ G_t \} _{t=0}^{\binom{n}{2}}$ be a random graph process on $[n]$, and let $g(n)$ be a function that tends to infinity with $n$ arbitrarily slowly. Then, with high probability, for every $i\in \mathbb{N}^+$,
\begin{align*}
g(n)^{-1}\cdot n^{2/3}\le  t_i^* \le  g(n) \cdot n^{2/3} \hspace{5mm} \text{ and } \hspace{5mm} \left| t_i-n\cdot \left(\frac{1}{6}\log n + \log \log n \right) \right| \leq g(n) \cdot n,
\end{align*}
and
\begin{enumerate}[label=(\roman*)]

\item $\mu (G_t) = n_0(G_t) + \lceil \frac{1}{2} n_1 (G_t) \rceil$ for  $t \le g(n)^{-1}\cdot n^{2/3}$; \label{main2:part1}

\item $\mu (G_t) = n_0(G_t) + \lceil \frac{1}{2} (n_1 (G_t)+i) \rceil$ for every $t_i^* \le t < t_{i+1}^*$; \label{main2:part2}

\item $\lim_{n\to \infty}\left( \mu (G_t) - n_0(G_t) - \lceil \frac{1}{2} n_1 (G_t) \rceil \right) =\infty$
for $g(n) \cdot n^{2/3} \leq t \leq  n\cdot \left(\frac{1}{6}\log n + \log \log n -g(n)\right)$; \label{main2:part3}

\item $\mu (G_t) =   n_0(G_t) + \lceil \frac{1}{2} (n_1 (G_t)+(i-1)) \rceil  $ for every $t_{i} \le t < t_{i-1}$; \label{main2:part4}

\item $\mu (G_t) =  n_0(G_t) + \lceil \frac{1}{2} n_1 (G_t)\rceil $ for every $t\geq t_1$. \label{main2:part5}

\end{enumerate}
\end{thm}

Here it is worth stressing that it is not enough to show that, for example, \ref{main2:part5} holds for $t=t_1$, since both the property of not spanning $i$ copies of a $3$-spider and the property of satisfying the equality $\mu (G_t) =n_0(G_t)+\lceil \frac{1}{2} n_1 (G_t)\rceil$ are not monotone properties.

The following corollary is a hitting time statement that can be derived from Theorem \ref{main2} \ref{main2:part5}.

\begin{corol*}
Let $\{ G_t \} _{t=0}^{\binom{n}{2}}$ be a random graph process on $[n]$ and let $k =k(n)\ll \frac{n^{2/3}}{\log n}$ be an integer. Then, with high probability, the hitting time of the monotone property $\mu (G_t) \le k$ is equal to the hitting time of the property $n_0(G_t)+\lceil \frac{1}{2} n_1 (G_t)\rceil \le k$.
\end{corol*}

\begin{proof}
An application of Chebychev's inequality gives that if $t'=(\frac16 \log n+\log\log n+h(n)) \cdot n$, where $h(n)=o(\log\log n)$, then $ n_0(G_{t'}) +\lceil \frac{1}{2} n_1 (G_{t'})\rceil$ is concentrated around its expectation, i.e., with high probability
$$
n_0(G_{t'}) +\lceil \frac{1}{2} n_1 (G_{t'})\rceil =(1+o(1))\cdot \mathbb{E}\left( n_0(G_{t'}) +\lceil \frac{1}{2} n_1 (G_{t'})\rceil \right)=(1+o(1))\cdot \frac{n^{2/3}e^{-2h(n)}}{3\log n}.
$$
By Theorem \ref{main2} we have that $t'>t_1$ with high probability, and since the sequence  $\{ n_0(G_t)+\lceil \frac{1}{2} n_1 (G_t)\rceil \}_{t=0}^{\binom{n}{2}}$ is non-increasing with respect to $t$, the following holds. For $k \ll \frac{n^{2/3}}{\log n}$ the hitting time of the property $n_0(G_t)+\lceil \frac{1}{2} n_1 (G_t)\rceil \le k$, denoted by $\sigma_k$, occurs after time $t_1$ with high probability. In addition, since $\mu (G_t) \ge n_0(G_t)+\lceil \frac{1}{2} n_1 (G_t)\rceil$, we have that  $\mu (G_t)>k$ for $t<
\sigma_k$. Combining the above with Theorem \ref{main2} \ref{main2:part5} we get the following corollary.
\end{proof}

In fact, Theorem \ref{main2} \ref{main2:part5} implies that if $0\le k_0 \ll \frac{n^{2/3}}{\log n}$ then, with high probability, the hitting times of $\mu (G_t) \le k$ and $n_0(G_t)+\lceil \frac{1}{2} n_1 (G_t)\rceil \le k$ are equal for all $0 \le k \le k_0$. For $k=1$ we get that the hitting time of containing a Hamilton path is equal with high probability to that of $n_0(G_t)+\lceil \frac{1}{2} n_1 (G_t)\rceil \le 1$, which is equal with high probability to the hitting time of $n_1(G_t) \le 2$.

Corollary \ref{corol2}, stated below, complements Theorem \ref{thm:main1} in the range $np\geq 0.4\log n$. It follows from Theorem \ref{main2} \ref{main2:part5} by the following two facts that hold for $G\sim G(n,p)$. First, if $p=p(n)$ is such that $np=\omega(1)$ and $np=o(n)$, then the number of edges of $G$ lies in $[\binom{n}{2}p-n, \binom{n}{2}p+n]$ with high probability. Second, conditioned on $G$ having exactly $t$ edges we have that $G$ has the same distribution as $G_t$.

\begin{corol*} \label{corol2}
Let $G\sim G(n,p)$, where $np-\frac13 \log n-2\log\log n \to \infty$. Then $\mu (G) = n_0(G) + \lceil \frac{1}{2}n_1(G) \rceil$ with high probability.
\end{corol*}

Our third main result in this paper deals with the completion number of pancyclicity, where a graph $G$ is said to be pancyclic if it contains a cycle of every length between 3 and $|V(G)|$.
Let $\hat{\mu}(G)$ denote the completion number of $G$ with respect to pancyclicity. Some connections between $\hat{\mu}(G)$ and $\mu (G)$ can already be derived from known results. Cooper and Frieze \cite{pancyclic} showed that the sharp threshold of Hamiltonicity in $G(n,p)$ is also a sharp threshold of pancyclicity (and, in fact, that the hitting times of the two properties in a random graph process are with high probability equal). As a consequence we have that $\hat{\mu}(G(n,p))=\mu (G(n,p))=0$ with high probability if $np-\log n -\log \log n \to \infty$.

In sparser random graphs, we can get an approximation of $\hat{\mu}(G(n,p))$ by using a deterministic bound on the difference  $\hat{\mu}(G)-\mu (G)$. It is known (see e.g. \cite{PANBOOK}, Section 4.5) that, for every $n$, there exists a pancyclic $n$-vertex graph with $n+(1+o(1)) \cdot \log_2 n$ edges. In other words, one can always add $(1+o(1))\cdot \log_2 n$ edges to a Hamiltonian $n$-vertex graph to make it pancyclic. In particular we get $\hat{\mu}(G) = \mu (G) + O(\log n)$ for every $n$-vertex graph $G$. As an immediate consequence of this, Theorem \ref{thm:main1} and Corollary \ref{corol2}, we get that, for $G\sim G(n,p)$ with $np \ge 20$ not too large, we have $\hat{\mu} (G)=(1+o(1)) \cdot \mu (G)$ with high probability. This approximation holds as long as $\mu (G)$ is typically much larger than $\log n$, which remains true until fairly close to the Hamiltonicity threshold.

Our final theorem in the paper shows that, typically, the two completion numbers are in fact equal. In other words, in terms of the number of edges needed to be added, pancyclicity is typically as cheap as Hamiltonicity in $G(n,p)$.

\begin{thm}\label{main3}
Let $np \ge 20$ and $G\sim G(n,p)$. Then, with high probability, $\hat{\mu} (G) =\mu (G)$.
\end{thm}

This, together with Theorem \ref{thm:main1} and Corollary \ref{corol2}, provides a characterization of $\hat{\mu} (G(n,p))$ for all values $p \ge 20/n$.

Our final result in this paper revolvs around the algorithmic aspect of completing a random graph $G$ to Hamiltonicity and pancyclicity
.

Note that $\mu (G) = 0$ if and only if $G$ is Hamiltonian. Thus, the problem of determining $\mu (G)$ is NP-complete, as it solves the Hamiltonicity problem. The same goes for the problem of determining $\hat{\mu}(G)$, since the problem of deciding pancyclicity is also known to be NP-complete.

In the realm of random graphs, algorithms for finding Hamilton cycles in polynomial time with high probability are known to exist (see \cite{CRE, ANASALG, BFF}). The following theorem shows that, in fact, determining with high probability $\mu (G)$ and $\hat{\mu}(G)$, and finding a minimum size completing set, can also be done by a polynomial time algorithm.

\begin{thm}\label{main4}
There is a polynomial time algorithm that, given a graph $G$, returns a set $F$ of non-edges of $G$. If $np \ge 20$ and $G\sim G(n,p)$ then, with high probability, $|F| = \mu (G)$ and $G \cup F$ is pancyclic.
\end{thm}

\paragraph{Paper structure} In Section \ref{sec-per} we provide notation and auxiliary results to be used in the rest of the paper. In Section \ref{sec:4core} and Section \ref{sec-identify} we construct tools for proving our theorems, and prove some key lemmas. Theorem \ref{thm:main1} is then proved in Section \ref{sec:proof1}, Theorem \ref{main3} and Theorem \ref{main4} in Section \ref{sec:proof3} and Theorem \ref{main2} in Section \ref{sec:proof2}.

\section{Preliminaries} \label{sec-per}

\subsection*{Notation}

We will use the following graph theoretic notation throughout the paper.

Given a graph $G$, the degree of a vertex $v\in V(G)$ in $G$, denoted by $d_G(v)$, is the number of edges of $G$ incident to $v$. We denone by $\Delta(G)$ the maximum degree of $G$. If $S\subseteq V(G)$ we denote by $d_G(v,S)$ the degree of $v$ into $S$, that is, the number of edges in $E(G)$ with one end in $S$, and the other end equal to $v$.
The (external) neighbourhood of a vertex subset $U$, denoted by $N_G(U)$, is the set of vertices in $V\setminus U$ adjacent to a vertex of $U$. For a vertex $v\in V(G)$ we further define $N_G^k(v)$, $N_G^{\le k}(v)$ and $N_G^{<k}(v)$ as the sets of vertices in $V(G)$ whose distance from $v$ is exactly $k$, at most $k$ and less than $k$, respectively. For $S\subseteq V(G)$ and an edge set $E$ spanned by $V(G)$ we let $G[S]$ be the subgraph of $G$ induced by $S$ and $G\cup E$ the graph $(V(G),E(G)\cup E)$. While using the above notation we occasionally omit $G$ if the identity of the graph $G$ is clear from the context.

For $i\in \left[ 0 ,|V(G)|-1 \right]$ we denote $n_i(G) \coloneqq \left| \{ v\in G \mid d_G(v) = i\} \right|$, the number of vertices in $G$ with degree $i$.

We denote by $\cL (G)$ the set $\{\ell\in [3,|V(G)|] \mid G\text{ contains a cycle of length }\ell\}$, and by $L_{\max}(G)\coloneqq \max (\cL (G))$ the length of a longest cycle in $G$.

We occasionally suppress rounding signs to simplify the presentation.

\subsection*{Auxiliary results}

Below is a list of auxiliary results which we will employ in our proof.

\begin{lemma}[Asymptotic equivalence of $G(n,p)$ and $G_t$, see e.g. \cite{FK}, Lemma 1.2] \label{lemma:equiv}
Let $\{ G_t \}_{t=0}^{\binom{n}{2}}$ be a random graph process on $[n]$. 
Let $\mathcal{P}$ be a graph property, $t=t(n)$ is such that $t\to \infty$ and $\binom{n}{2}-t \to \infty$ and $p=t/\binom{n}{2}$.
Then
$$
\lim _{n\to \infty}\pr \left( G_t\in \mathcal{P} \right) \le 10t^{1/2}\cdot \lim _{n\to \infty}\pr \left( G(n,p)\in \mathcal{P} \right) .
$$
\end{lemma}

\begin{thm}[Cycle lengths in sparse random graphs, see Theorem 1.2 and Lemma 1.3 of \cite{ANAS} and Equation (2) of \cite{AF}]\label{thm:ANAS}
Let $p=\frac{c}{n}$, $c=c(n)\geq 20$ and $G\sim G(n,p)$. Then with high probability
$$
\Lmax (G) = \left( 1-(c+1)e^{-c}-c^2e^{-2c}+O(c^6e^{-3c}) \right) \cdot n.
$$
In particular with high probability $\Lmax(G)\geq n-0.04c^3e^{-c}n$. Additionally, if $\ell _0 = \ell _0(n) \to \infty$ then with high probability $\left[ \ell _0,\Lmax \right] \subseteq \mathcal{L}(G).$
\end{thm}

\begin{thm}[Threshold of pancyclicity, see \cite{pancyclic}]\label{thm:CF}
Let $p=p(n)$ and $G\sim G(n,p)$. Then
$$
\lim _{n\to \infty} \pr \left( G \text{ is pancyclic} \right) = \lim _{n\to \infty} \pr \left( \delta (G) \ge 2 \right) .
$$
In particular, if $np-\log n-\log \log n \to \infty$ then $G$ is pancyclic with high probability.
\end{thm}

\begin{thm}[Near-pancyclicity in dense random graphs \cite{LUC91}]\label{thm:LUC}
Let $\varepsilon >0$, let $p=p(n)$ be such that $np \to \infty$, and let $G\sim G(n,p)$. Then with high probability $\left[ 3,n-(1+\varepsilon )\cdot n_1(G) \right] \subseteq \cL (G)$.
\end{thm}

Lemma \ref{lemma:martingales} is obtained from \cite{WAR} Theorem 1.2 by taking $c_i = d, d_i = N, \gamma _i = N^{-1}$ for all $i$.

\begin{lemma}[Vertex martingales]\label{lemma:martingales} 
Let $G\sim G(n,p)$, let $f$ be a graph theoretic function, and let $X_0,X_1,...,X_n$ be the corresponding vertex exposure martingale. Further assume that there is a graph property $\mathcal{P}$ and positive integers $d,N$ such that, for every pair $G_1,G_2$ of graphs on $V$ such that $E(G_1)\triangle E(G_2) \subseteq \{v\}\times V$ for some $v\in V$, the following holds:
\begin{eqnarray*}
|f(G_1)-f(G_2) | \le
	\begin{cases}
        d & \text{if } G_1,G_2\in \mathcal{P} ;\\
        N & \text{otherwise} .
    \end{cases}
\end{eqnarray*}
Then
$$
\pr \left( | X_n - X_0 |<  t \right) \le 2\exp \left( -\frac{t^2}{2n(d+1)^2} \right) +n\cdot N\cdot \pr (G\notin \mathcal{P}) .
$$

\end{lemma}

\section{The strong 4-core and its structure} \label{sec:4core}

\begin{defin}
The \emph{strong $k$-core} of a graph $G$ is the maximal subset $S\subseteq V(G)$ such that $d_G(v,S)\ge k$ for every $v\in S\cup N_G(S)$.
\end{defin}

In our knowledge, the concept of the strong $k$-core was first used by Anastos and Frieze in \cite{AF} for identifying a longest cycle in $G(n,p)$ and formally defined in \cite{ANAS}. Note that, since the property $d_G(v,S)\ge k$ for every $v\in S\cup N_G(S)$ is closed under union, there is indeed a maximal subset in $V(G)$ satisfying it, and therefore the strong $k$-core is well-defined.

For the purpose of our proofs, we are specifically interested in the strong 4-core of $G(n,p)$, and its typical properties. Henceforth, we denote by $C(G)$ the vertex set of the strong $4$-core of a graph $G$, by $B(G)$ the ``border" set $N_G(C(G))$ and by $A(G)$ the set $V(G)\setminus (B(G)\cup C(G))$. We will also denote by $G^{AB}$ the induced subgraph of $G$ on the set $A(G)\cup B(G)$. The sets $A(G),B(G),C(G)$ can be identified using the following red/blue/black colouring process. Initially colour every vertex of $G$ black. While there exists a black or blue vertex with fewer than $4$ black neighbours, recolour that vertex red and its black neighbours blue.  The sets $A(G),B(G)$ and $C(G)$ correspond to the sets of red, blue and black vertices at the end of the colouring procedure, respectively. By studying the colouring process one can prove Lemma \ref{lemma:4coreProperties}, stated shortly. In this lemma we gather a number of typical properties of the strong $4$-core. The proof for the lemma is located in Appendix \ref{app:sec:lemma:4coreProperties}. Before stating it we lay some final definitions related to the strong $4$-core. First, let $S(G)$ be the set of connected components in $G^{AB}$ that have one vertex in $A(G)$ and three vertices in $B(G)$. Second, for $k\geq 2$ let $\cE_k(G)$ be the event that $G$ contains a $k$-cycle or $G^{AB}$ contains a tree component on $k+1$ vertices $b,a_1,...,a_k$, with $b\in B(G)$ and $\{a_1,a_2,a_3,...,a_k\}\subseteq A(G)$, where $b$ is adjacent to $a_1,a_2$ and $a_3$ and $a_3a_4...a_k$ is a path.  Next, set $\cE(G):=\cup_{k=3}^{\log\log n} \cE_k(G)$. Finally, denote by $s_3(G)$ the number of $3$-spiders of $G$
%, that is, $s_3(G)$ is the number of copies in $G$ of the graph spanned by $3$ vertices with degree $2$ that share a common neighbour, and their neighbours 
(recall Figure \ref{fig:3-spider}).

The following lemma identifies properties typical to $G^{AB}$ which will be useful for our proofs.

\begin{lemma}\label{lemma:4coreProperties} 
Let $d=d(n)\coloneqq np$ be such that $20\leq d \leq \log n+1.1\log\log n$, and let $G\sim G(n,p)$. For $i\ge 1$ denote by $X_i$ the number of vertices in $A(G)\cup B(G)$ whose connected component in $G^{AB}$ contains $i$ vertices in $A(G)$. Also let $p'=\frac{\log n+18\log\log n }{9n}$. Then,
\begin{enumerate}[label=(\roman*)]
\item $\mathbb{E}(X_i) \le \max \left\{ \frac{(2ed)^{4i}e^{-id}}{15id}\cdot n, n^{-3}\right\}$ for every $1\le i\le n$; \label{4core:sizeX_i} 
\item $G^{AB}$ has no component with more than $\log\log n$ vertices and at least $1$ cycle;
\label{4core:notrees} 
\item $|S(G)|\geq 0.1d^3 e^{-d}\cdot n $ with high probability; \label{4core:sizeofAandS} 
\item the event $\cE(G)$ occurs with high probability; \label{4core:eventE}
\item if furthermore $p\leq p'$, then  $s_3(G)\geq 10^{-8}\cdot n^{2/3}$ with probability $1-o(n^{-2})$. \label{4core:3spiders}
\end{enumerate}

\end{lemma}

The following lemma, which was proven by the second author in \cite{ANAS}, describes a strong Hamiltonicity property that the strong $4$-core satisfies. 

\begin{lemma}[Strong Hamiltonicity of $B(G)\cup C(G)$] \label{lemma:4coreHmty}
Let $ 20\leq c$ and $G\sim G(n,c/n)$. Let $U\subseteq B(G)$, let $M \subseteq \binom{U}{2}$ be a matching on the set $U$ that is allowed to contain pairs not in $E(G)$, and $H\coloneqq G\left[ C(G) \cup U \right] \cup M$. Then, with probability $1-O(n^{-2})$, $H$ contains a Hamilton cycle that spans all the edges of $M$.
\end{lemma}
We emphasize that the set of edges $M$, in the lemma above, does not need to be a subset of the edges of $G$. 
\section{Identifying the completion number of \texorpdfstring{$G(n,p)$}{Lg}  with respect to Hamiltonicity} \label{sec-identify}

Towards estimating $\mu (G)$, we first introduce the function $\mu '(G)$ stated shortly. Then we prove that, if $G\sim G(n,p)$ with $np \ge 20$, then $\mu (G) = \mu '(G)$ with high probability.

Let $G$ be a graph. We define $\mu '(G)$ as follows. Let $\mathcal{Q}$ be the set of all disjoint path covers of $G^{AB}$. For a disjoint path cover $Q$ of $G^{AB}$, let $a(Q)$ be the number of vertices in $A(G)$ that are endpoints of paths in $Q$. Here, a vertex of $A(G)$ that constitutes a path of length 0 in $Q$ is counted twice towards $a(Q)$, as both the start of the path and its end. Define $a(G) \coloneqq \min _{Q\in \mathcal{Q}}\left( a(Q)\right)$ and $\mu '(G) \coloneqq \lceil \frac{1}{2}a(G) \rceil$.

\begin{observation} \label{obs:mu}
The inequality $\mu (G) \ge \mu '(G)$ holds for every graph $G$.
\end{observation}
Indeed, if $F$ is a set of edges of size $\mu(G)$ such that $G\cup F$ spans a Hamilton cycle, then, for every Hamilton cycle $H$ in $G\cup F$, the set of edges $E(H)\setminus F$ forms a disjoint path cover $P$ of size $\mu(G)$.
Thereafter, $Q \coloneqq P\cap G^{AB}$ is a disjoint path cover of $G^{AB}$. Since $E_G(A(G),C(G)) = \emptyset$, any endpoint in $A(G)$ of a path in $Q$ must also be an endpoint of a path in $P$. In other words, there are at least $a(Q)$ endpoints in $P$, and therefore at least $\lceil \frac{1}{2}a(Q) \rceil$ paths in $Q$. By definition, $\lceil \frac{1}{2}a(Q) \rceil \ge \mu '(G)$, and therefore overall $\mu (G) \ge \lceil \frac{1}{2}a(Q) \rceil \ge \mu '(G)$.

Next we show that if $G\sim G(n,p)$ then also $\mu (G) \le \mu '(G)$ with high probability. In fact, we will prove a stronger claim, which will be useful for proving Theorem \ref{main3}. Recall the definitions of the sets $A(G),B(G),C(G),S(G)$ and the event $\cE (G)$.
Additionally, let $\cE_1(G)$ be the event that either $G^{AB}$ contains at least two two-vertex components that contain
a vertex from $A(G)$ and a vertex from $B(G)$, or $G^{AB}$ contains no component with at least two vertices in $A(G)$ and no isolated vertices.
As preparation for our proof we require the following lemma, which is proved in Appendix \ref{sec:app:4coreProperties2}.

\begin{lemma}\label{lem:4coreProperties2} 
Let $20\leq np \leq \log n+1.1\log\log n$ and $G\sim G(n,p)$. Then, with high probability the event $\cE_1(G)$ occurs. 
\end{lemma}

We now turn to prove the following main lemma, which will be an important ingredient in proving Theorem \ref{thm:main1} and Theorem \ref{main3}.

\begin{lemma} \label{lemma:mu}
Let $20 \leq np \leq  \log n+1.1\log\log n$ and $G\sim G(n,p)$, and denote $s\coloneqq |S(G)|$. Then with high probability there is a set $F\subseteq \binom{V(G)}{2}$ of size $\mu '(G)$ such that the following hold.
\begin{enumerate}[label=(\roman*)]
\item  $G\cup F$ is Hamiltonian; \label{lemma:mu-hamiltonian}
\item  $n-\ell\in \cL (G\cup F)$ for $1 \le \ell \le s$; \label{lemma:mu-longcycles}
\item  $\cL (G\cup F)$ contains all the integers in $[3,\log\log n]$. \label{lemma:mu-shortcycles}
\end{enumerate}
\end{lemma}

Observe that in particular, given Observation \ref{obs:mu}, Lemma \ref{lemma:mu}\ref{lemma:mu-hamiltonian} implies that $\mu (G) = \mu '(G)$ holds with high probability.

We first present a proof that there is a set $F$ that satisfies \ref{lemma:mu-hamiltonian}, since this proof is much less involved, and constitutes a proof that $\mu (G) = \mu '(G)$ with high probability. In order for $F$ to satisfy \ref{lemma:mu-longcycles} and \ref{lemma:mu-shortcycles} as well, the set $F$ should be constructed more carefully. We provide the full proof of the lemma immediately after the proof of \ref{lemma:mu-hamiltonian} alone, so the reader may choose to skip directly to it.

\begin{proof}[Proof of \ref{lemma:mu-hamiltonian}]
By Lemma \ref{lemma:4coreProperties} and Lemma \ref{lem:4coreProperties2} we may assume that the event $\cE_1(G)$ and the events listed in Lemma \ref{lemma:4coreProperties} all occur. 

Let $Q$ be a disjoint path cover of $G^{AB}$ with $\lceil \frac{1}{2}a(Q) \rceil=\mu '(G)$, such that $Q$ has no edges contained entirely in $B(G)$ (observe that we may assume that such a cover $Q$ exists, since removing and adding edges that are contained in $B(G)$ to $Q$ does not affect $a(Q)$). Let $B_M$ be the set of all vertices of $B(G)$ that are endpoints of paths of $Q$ of positive length, $B'(G)$ be the set of vertices in $B(G)$ that are not internal vertices of paths in $Q$ 
and $M$ be a near-perfect matching on $B_M$.

By Lemma \ref{lemma:4coreHmty}, $G\cup M$ contains a cycle $H$ that covers $B'(G)\cup C(G)$, and contains all the edges of $M$, with probability $1-O(n^{-2})$. Assume that this is the case, and denote by $P$ the subgraph obtained by taking the edges of $H\setminus M$ along with all the edges of paths of $Q$ (see Figure \ref{fig:proof-of-(i)}). If $M$ is not a perfect matching on $B_M$, that is, if $M$ leaves out a single vertex $b\in B_M$, then remove from $P$ the single edge in $Q$ incident to $b$.

\begin{figure}[h]
\centering
\begin{tikzpicture}[vertex/.style={draw,circle,color=black,fill=black,inner sep=1,minimum width=4pt},scale=1]
    
    \draw[thick] (0,0) rectangle (12,5.5);
	\draw[thick, dotted] (4,0) rectangle (8,5.5);
	\draw[thick, dotted] (4,4.7) rectangle (8,5.5);
	
	\node at (2,-0.5) {$A(G)$};
	\node at (6,-0.5) {$B(G)$};
	\node at (6,2.5) {$B'(G)$};
	\node at (6.5,5.1) {$B(G)\setminus B'(G)$};
	\node at (10,-0.5) {$C(G)$};
	\node at (-1,2.5) {$V(G)$};
	
	\draw[rounded corners=0.7cm] (4.5, 0.5) rectangle (11.5, 4.5) {};
	\node at (11,2.5) {$H$};
	
	\fill[white] (4.5,1.4) circle (0.3);
	\node[vertex] (b1) at (4.5,1.1) {};
	\node[vertex] (b2) at (4.5,1.7) {};
	\draw[very thick, dotted] (b1) to (b2);
	\node[vertex] (a1) at (2.2,1.1) {};
	\node[vertex] (a2) at (3.3,1.7) {};
	\draw[thick, snake it] (a1) to (b1);
	\draw[thick, snake it] (a2) to (b2);
	
	\fill[white] (4.5,2.5) circle (0.3);
	\node[vertex] (b3) at (4.5,2.2) {};
	\node[vertex] (b4) at (4.5,2.8) {};
	\draw[very thick, dotted] (b3) to (b4);
	\node[vertex] (a3) at (3,2.2) {};
	\node[vertex] (a4) at (2.5,2.8) {};
	\draw[thick, snake it] (a3) to (b3);
	\draw[thick, snake it] (a4) to (b4);
	
	\fill[white] (4.5,3.6) circle (0.3);
	\node[vertex] (b5) at (4.5,3.3) {};
	\node[vertex] (b6) at (4.5,3.9) {};
	\draw[very thick, dotted] (b5) to (b6);
	\coordinate (mid) at (2.7,3.6) {};
    \draw[thick, out=180,in=270, snake it] (b5) to (mid)
    [thick, out=90,in=180, snake it] (mid) to (b6);
    
    \node[vertex] (a5) at (1,0.8) {};
	\node[vertex] (a6) at (1,2.7) {};
	\draw[thick, snake it] (a5) to (a6);
	
	\node[vertex] (b7) at (4.5,5.1) {};
	\node[vertex] (a7) at (3,4.6) {};
	\node[vertex] (a8) at (1.5,5.1) {};
	\draw[thick, snake it] (a7) to (b7);
	\draw[thick, snake it] (a8) to (b7);
	
	\node[vertex] (a9) at (1.5,4) {};
	
\end{tikzpicture}
\caption{\textit{The disjoint path covering $P$, all of whose endpoints (except for maybe one) lie in $A(G)$. The dotted edges represent the edges of the matching $M$, and the snaked lines are the paths of $Q$.}}\label{fig:proof-of-(i)}
\end{figure}

Observe that $P$ is a union of a cycle and paths if $M=\emptyset$, and otherwise it is a union of vertex disjoint paths. 

First assume that $P$ contains a cycle. Then $Q$ contains at most one path with one endpoint in $A(G)$ and one in $B(G)$, which implies that $G^{AB}$ does not have two components with one vertex in $A(G)$ and one in $B(G)$. Thus, since $\cE_1(G)$ occurs, this means that $G^{AB}$ contains no component with two vertices in $A(G)$, and no isolated vertices. In this case, $P$ is either a cycle if $B'(G)$ is empty, or a cycle and a path that was created by removing an edge incident to the single vertex in $B'(G)$ that was unmatched by $M$. In the first case $P$ is a Hamilton cycle in $G$ and $F=\emptyset$ satisfies the requirements. In the second case, putting the removed edge back and removing a cycle edge instead yields a Hamilton path in $G$, and setting $F$ as the singleton containing the edge between its two endpoints satisfies the requirements.
 
Now assume the complement case, that is, $P$ is a disjoint path covering of $G$. Observe that all the vertices of $C(G)$ are internal vertices in $P$. Indeed, the vertices of $C(G)$ all have degree 2 in $H$, and since $M$ does not match any vertex in $C(G)$, their degree is 2 in $H\setminus M$ as well, and therefore in $P$. Additionally, all vertices of $B(G)$ are also internal vertices in $P$. Indeed, if $v\in B'(G)$ constitutes a path of length 0 in $Q$, or is the at most one unmatched endpoint of a positive length path, then it is not matched by $M$, and therefore had degree 2 in $H\setminus M$, and therefore in $P$. Otherwise, if $v\in B'(G)$ is matched by $M$ then it has degree 2 in $P$, since it has one neighbour in $B(G)\cup C(G)$ and one neighbour in $A(G)$. The last case is $v\in B(G)\setminus B'(G)$, in which $v$ is internal in $Q$ (and cannot have been incident to the additional removed edge), and therefore also internal in $P$.

Now, $P$ is a disjoint path covering of $G$ such that all its endpoints, except for at most one, are endpoints of $Q$ in $A(G)$ (if we removed an additional edge from an unmatched endpoint, its neighbour in $Q$ became an additional endpoint in $P$). It follows that the number of paths in $P$ is $\lceil \frac{1}{2}a(Q) \rceil = \mu '(G)$, and therefore $\mu '(G) \ge \mu (G)$, and a set $F$ as desired exists.
\end{proof}

We now present the proof of the full version of the lemma. Similarly to the proof above, we will find a disjoint path covering of $G$ with $\mu '(G)$ paths, and $F$ will be an edge set that connects the endpoints of the paths in this cover. However, in order to also satisfy \ref{lemma:mu-longcycles} and \ref{lemma:mu-shortcycles}, the path covering, and subsequently $F$, will be constructed more carefully.

\begin{proof}[Full proof of Lemma \ref{lemma:mu}] 
As in the proof of \ref{lemma:mu-hamiltonian} assume that the event $\cE_1(G)$ and the events listed in Lemma \ref{lemma:4coreProperties} all occur, and let $Q$ be a disjoint path cover of $G^{AB}$ with $\lceil \frac{1}{2}a(Q) \rceil=\mu '(G)$, such that $Q$ has no edges contained entirely in $B(G)$.

We will construct the desired set $F$ in three parts, that is, we construct three edge sets $F_0,F_1,F_2$ such that $F\coloneqq F_0 \cup F_1 \cup F_2$ is of size $\mu '(G)$ and satisfies \ref{lemma:mu-hamiltonian}--\ref{lemma:mu-shortcycles}. Let $K$ be the set of integers in $[3,\log\log n]$ such that  
$G^{AB}$ spans a tree component $C_k$  on a vertex $b^k\in B(G)$ and $k$ vertices $\{a_1^k,a_2^k,a_3^k,...,a_k^k\}$ in $A(G)$ where $b^k$ is adjacent to $a_1^k,a_2^k$ and $a_3^k$ and $a_3^ka_4^k...a_k^k$ is a path.

First, we modify $Q$ slightly and define $F_0$. For $k\in K$, as $C_k$ has $3$ vertices of degree $1$ that lie in $A(G)$, there exist at least $2$ paths in $Q$ that cover these $3$ vertices. If none of $a_1^k$ and $a_2^k$ corresponds to a path in $Q$, then $Q$ contains the path $a_1^kb^ka_2^k$ and a path of the form $a^k_ja_{j+1}^k...a_k^k$. These two paths have $4$ endpoints in $A(G)\cap V(C_k)$. If exactly one of $a_1^k$ and $a_2^k$ corresponds to a path in $Q$, then the paths in $Q$ spanned by $C_k$ have at least $2$ endpoints in $A(G)\cap V(C_k)$, other than this vertex. The third case is that both  $a_1^k$ and $a_2^k$ correspond to isolated paths in $Q$. 
In all three cases, we may replace the paths that in $Q$ that cover $C_k$ with the paths $a_2^kb^ka_3^ka_4^k...a_k^k$ and $a_1^k$ as doing so does not increase the value of $a(Q)$. We then remove from $Q$ these 2 paths and add the path $a_2^kb^ka_3^ka_4^k...a_k^ka_1^k$. We let $F_0$ be the set of edges $\{a_k^ka_1^k:k\in K\}$ (see Figure \ref{fig:F0}), and denote by $Q^*$ the path covering of $G^{AB}\cup F_0$ we obtained by the above process. 

Denote by $Q_1^*$ the set of paths in $Q^*$ that have one endpoint in $A(G)$ and one endpoint in $B(G)$. Additionally, if $|Q_1^*|$ is odd, let $v$ be a vertex in $B(G)$ that constitutes a path of length 0 in $Q^*$. We add $v$ to $Q_1^*$ and henceforward we use the convention that the path $v$ has $2$ endpoints, one in each of $A(G),B(G)$.  {Here we are using that Lemma \ref{lemma:4coreProperties} \ref{4core:sizeofAandS} implies that $|S(G)|\geq 1$ and every component in $S(G)$ is covered by two paths in $Q^*$, a path of length $2$ and a path of length $0$ that consists of a vertex in $B(G)$.}
So $Q^*$ indeed contains a path of length $0$ whose single vertex belongs to $B(G)$. Note that now $|Q_1^*|$ is even. Furthermore, $\frac{1}{2}a(Q^*)+|F_0|=\lceil \frac{1}{2}a(Q)\rceil$.

\begin{figure}[h]
\centering
\begin{tikzpicture}[vertex/.style={draw,circle,color=black,fill=black,inner sep=1,minimum width=4pt},scale=1]

	\coordinate (l1) at (-1.8,0.15) {};
	\coordinate (l2) at (-1.3,0.15) {};
	\coordinate (l3) at (1.3,0.15) {};
	\coordinate (l4) at (1.8,0.15) {};
	\node at (-2.2,0.5) {$A(G)$};
	\node at (-2.2,-0.2) {$B(G)$};
	
	\draw[dashed] (l1) to (l2);
	\draw[] (l2) to (l3);
	\draw[dashed] (l3) to (l4);
	
	\node[vertex] (b) at (0,0) {};
	\node at (0,-0.4) {$b^k$};
	
	\node[vertex] (a1) at (-0.7,1.3) {};
	\node[vertex] (a2) at (0,0.6) {};
	\node[vertex] (a3) at (0.5,0.6) {};
	\node at (-1.1,1.3) {$a_1^k$};
	\node at (0,1) {$a_2^k$};
	\node at (0.8,0.4) {$a_3^k$};
	
	\node[vertex] (a4) at (1,1.2) {};
	\node at (1.3,1) {$a_4^k$};
	
	\coordinate (a5) at (1.15,1.38) {};
	\coordinate (ak-2) at (1.6,1.92) {};
	
	\node[vertex] (ak-1) at (1.75,2.1) {};
	\node[vertex] (ak) at (2.25,2.7) {};
		\node at (2.25,1.9) {$a_{k-1}^k$};
	\node at (2.55,2.5) {$a_k^k$};
	
	\draw[thick, dotted] (b) to (a1);
	\draw[thick] (b) to (a2);
	\draw[thick] (b) to (a3);
	\draw[thick] (a3) to (a4);
	\draw[thick] (a4) to (a5);
	\draw[very thick, dotted] (a5) to (ak-2);
	\draw[thick] (ak-2) to (ak-1);
	\draw[thick] (ak-1) to (ak);

	\coordinate (arr1) at (3,1.3) {};
	\coordinate (arr2) at (4.3,1.3) {};
	\draw[ultra thick, arrows = {-Stealth[]}] (arr1) to (arr2);
	
	\coordinate (mid) at (7,0) {};

	\coordinate (nl1) at (5.2,0.15) {};
	\coordinate (nl2) at (5.7,0.15) {};
	\coordinate (nl3) at (8.3,0.15) {};
	\coordinate (nl4) at (8.8,0.15) {};
	\node at (4.8,0.5) {$A(G)$};
	\node at (4.8,-0.2) {$B(G)$};
	
	\draw[dashed] (nl1) to (nl2);
	\draw[] (nl2) to (nl3);
	\draw[dashed] (nl3) to (nl4);
	
	\node[vertex] (nb) at (7,0) {};
	\node at (7,-0.4) {$b^k$};
	
	\node[vertex] (na1) at (6.3,1.3) {};
	\node[vertex] (na2) at (7,0.6) {};
	\node[vertex] (na3) at (7.5,0.6) {};
	\node at (5.9,1.3) {$a_1^k$};
	\node at (7,1) {$a_2^k$};
	\node at (7.8,0.4) {$a_3^k$};
	
	\node[vertex] (na4) at (8,1.2) {};
	\node at (8.3,1) {$a_4^k$};
	
	\coordinate (na5) at (8.15,1.38) {};
	\coordinate (nak-2) at (8.6,1.92) {};
	
	\node[vertex] (nak-1) at (8.75,2.1) {};
	\node[vertex] (nak) at (9.25,2.7) {};
	\node at (9.25,1.9) {$a_{k-1}^k$};
	\node at (9.55,2.5) {$a_k^k$};
	
	\draw[thick, dotted] (nb) to (na1);
	\draw[thick] (nb) to (na2);
	\draw[thick] (nb) to (na3);
	\draw[thick] (na3) to (na4);
	\draw[thick] (na4) to (na5);
	\draw[very thick, dotted] (na5) to (nak-2);
	\draw[thick] (nak-2) to (nak-1);
	\draw[thick] (nak-1) to (nak);	
	
	\draw[ultra thick] (na1) to (nak);

\end{tikzpicture}
\caption{\textit{Modifying the covering of $C_k$ in $Q$ to obtain $Q^*$, and adding $\{a_k^k,a_1^k\}$ to $F_0$. Adding $F_0$ to $G$ will now result in a graph that contains a $k$-cycle.}}\label{fig:F0}
\end{figure}

Second, denote by $q_1^*$ the size of $Q_1^*$, and by $x_1,...,x_{q_1^*},y_1,...,y_{q_1^*}$ the set of endpoints of paths in $Q_1^*$, so that $x_i,y_i$ are the two endpoints of the same path, where $x_i \in A(G)$ and $y_i\in B(G)$. We then let $M$ be the matching $\{y_{2i-1}y_{2i}:1\leq i\leq  q_1^*/2\}$ and $F_1$ the matching $\{x_{2i-1}x_{2i}:2\leq i\leq  q_1^*/2\}$ (see Figure \ref{fig:F1}). Furthermore for $2\leq i\leq  q_1^*/2$ we let $P_i$ be the path $y_{2i-1}Q_{2i-1}x_{2i-1}x_{2i}Q_{2i}y_{2i-1}$ where $Q_j$ is the path from $y_j$ to $x_j$ in $Q^*$.

\begin{figure}[h]
\centering
\begin{tikzpicture}[vertex/.style={draw,circle,color=black,fill=black,inner sep=1,minimum width=4pt},scale=1]

	\draw[rounded corners=0.7cm] (0, 0) rectangle (4, 7.5) {};

	\fill[white] (0,1) circle (0.3);
	\node[vertex] (b1) at (0,0.7) {};
	\node[vertex] (b2) at (0,1.3) {};
	\draw[very thick, dotted] (b1) to (b2);
	\node[vertex] (a1) at (-1.5,0.7) {};
	\node[vertex] (a2) at (-1.5,1.3) {};
	\draw[thick, snake it] (a1) to (b1);
	\draw[thick, snake it] (a2) to (b2);
	\draw[very thick, densely dashed] (a1) to (a2);
	\node at (-2.1,1.3) {$x_{q_1^*-1}$};
	\node at (-2.1,0.7) {$x_{q_1^*}$};
	\node at (0.6,1.3) {$y_{q_1^*-1}$};
	\node at (0.6,0.7) {$y_{q_1^*}$};
	
	\fill[white] (0,2.1) circle (0.3);
	\node[vertex] (b3) at (0,1.8) {};
	\node[vertex] (b4) at (0,2.4) {};
	\draw[very thick, dotted] (b3) to (b4);
	\node[vertex] (a3) at (-2,1.8) {};
	\node[vertex] (a4) at (-2,2.4) {};
	\draw[thick, snake it] (a3) to (b3);
	\draw[thick, snake it] (a4) to (b4);
	\draw[very thick, densely dashed] (a3) to (a4);

	\coordinate (mid1) at (-0.6,2.9) {};
	\coordinate (mid2) at (-0.6,3.5) {};
	\draw[ultra thick, dotted] (mid1) to (mid2);

	\fill[white] (0,4.3) circle (0.3);
	\node[vertex] (b5) at (0,4) {};
	\node[vertex] (b6) at (0,4.6) {};
	\draw[very thick, dotted] (b5) to (b6);
	\node[vertex] (a5) at (-1.6,4) {};
	\node[vertex] (a6) at (-1.6,4.6) {};
	\draw[thick, snake it] (a5) to (b5);
	\draw[thick, snake it] (a6) to (b6);
	\draw[very thick, densely dashed] (a5) to (a6);

	\fill[white] (0,3.2) circle (0.4);
	\draw[ultra thick, dotted] (b4) to (b5);

	\fill[white] (0,5.4) circle (0.3);
	\node[vertex] (b7) at (0,5.1) {};
	\node[vertex] (b8) at (0,5.7) {};
	\draw[very thick, dotted] (b7) to (b8);
	\node[vertex] (a7) at (-1,5.1) {};
	\node[vertex] (a8) at (-1,5.7) {};
	\draw[thick, snake it] (a7) to (b7);
	\draw[thick, snake it] (a8) to (b8);
	\draw[very thick, densely dashed] (a7) to (a8);
	\node at (-1.4,5.7) {$x_3$};
	\node at (-1.4,5.1) {$x_4$};
	\node at (0.4,5.7) {$y_3$};
	\node at (0.4,5.1) {$y_4$};

	\fill[white] (0,6.5) circle (0.3);
	\node[vertex] (b9) at (0,6.2) {};
	\node[vertex] (b10) at (0,6.8) {};
	\draw[very thick, dotted] (b9) to (b10);
	\node[vertex] (a9) at (-2.3,6.2) {};
	\node[vertex] (a10) at (-2.3,6.8) {};
	\draw[thick, snake it] (a9) to (b9);
	\draw[thick, snake it] (a10) to (b10);
	\node at (-2.7,6.8) {$x_1$};
	\node at (-2.7,6.2) {$x_2$};
	\node at (0.4,6.8) {$y_1$};
	\node at (0.4,6.2) {$y_2$};
	\node at (-1.15,7.2) {$Q_1$};

	\coordinate (bord1) at (-0.3,-0.2) {};
	\coordinate (bord2) at (-0.3,7.7) {};
	\draw[thick, dotted] (bord1) to (bord2);
	\node at (-1.5,-0.4) {$A(G)$};
	\node at (1.7,-0.4) {$B(G)\cup C(G)$};

\end{tikzpicture}
\caption{\textit{The set $F_1$ (dashed), which corresponds to the matching $M$ (dotted). Adding $F_1$ to the graph incorporates the paths of $Q^*$ with one endpoint in $A(G)$ and one in $B(G)$ (snaked) into a long path that covers $C(G)$ and all the vertices in $B(G)$ that are not internal in $Q^*$, starting at $x_1$ and ending at $x_2$.}}\label{fig:F1}
\end{figure}

Finally, to construct $F_2$ let $Q^*_{AB}$ be the set of paths in $Q^*$ with both endpoints in $A(G)$.  {If $Q^*_{AB}=\emptyset$ then we let $F_2=\emptyset$. Else $Q^*_{AB}\neq \emptyset$ and $q_1^*\geq 2$. Indeed, $Q^*_{AB}\neq \emptyset$ only if $G^{AB}$ contains a component that spans at least $2$ vertices in $A$ or an isolated vertex. Thus, in the event $\cE_1$, only if $q_1^*\geq 2$.  In the case $Q^*_{AB}\neq \emptyset$ and $q_1^*\geq 2$} we order and orient the paths in $Q^*_{AB}$ arbitrarily and add to $F_2$ an edge between the ending of every path to the starting vertex of the next path. In addition we add to $F_2$ the edges between $y_1$ and the start of the first path and $y_2$ and the ending of the last path,  { here we are using that $q_1^*\geq 2$}. We then let $P_1$ be the $x_1$ to $x_2$ path in $G\cup F_2$ that spans all the paths in $Q^*_{AB}$.

%\begin{figure}[h]
%\centering
%\begin{tikzpicture}[vertex/.style={draw,circle,color=black,fill=black,inner sep=1,minimum width=4pt},scale=1]
%
%
%	
%	%\draw[rounded corners=0.7cm] (0, 0) rectangle (4, 4) {};
%	
%	\draw[thick, snake it] (2.5,3) circle (2.5);
%	
%	
%	\fill[white] (0,3) circle (0.3);
%	\coordinate (y2) at (0,2.7) {};
%	\coordinate (y1) at (0,3.3) {};
%	\node[vertex] (x2) at (-1.3,2.7) {};
%	\node[vertex] (x1) at (-1.3,3.3) {};
%	\draw[thick, snake it] (x1) to (y1);
%	\draw[thick, snake it] (x2) to (y2);
%	\node at (-1.3,3.7) {$x_1$};
%	\node at (-1.3,2.3) {$x_2$};
%	
%	\node[vertex] (a1) at (-2,4) {};
%	\node[vertex] (a2) at (-2.8,4.8) {};
%	\draw[thick, snake it] (a1) to (a2);
%	\draw[thick, dashed] (a1) to (x1);
%	
%
%	\node[vertex] (b1) at (-4,4.8) {};
%	\node[vertex] (b2) at (-4.8,4) {};
%	\draw[thick, snake it] (b1) to (b2);
%	\draw[thick, dashed] (a2) to (b1);
%	
%	
%	\node[vertex] (c2) at (-4,1.2) {};
%	\node[vertex] (c1) at (-4.8,2) {};
%	\draw[thick, snake it] (c1) to (c2);
%	\draw[thick, dashed] (b2) to (c1);
%	
%
%	\node[vertex] (d2) at (-2,2) {};
%	\node[vertex] (d1) at (-2.8,1.2) {};
%	\draw[thick, snake it] (d1) to (d2);
%	\draw[thick, dashed] (c2) to (d1);
%	\draw[thick, dashed] (d2) to (x2);
%	
%	\node at (-4,3) {$P_1$};
%	
%	
%\end{tikzpicture}
%\caption{The set $F_2$ (dashed), which connects the paths of $Q_{AB}^*$ and the long path between $x_1,x_2$ in $G\cup F_1$ (snaked) into a Hamilton cycle.}\label{fig:F2}
%\end{figure}

Set $F=F_0\cup F_1\cup F_2$. Note that the edges in $F_0$ are incident to the vertices $a_k^k, k\in K$. Each such vertex lies in the interior of some path in $Q^*$, i.e. it is not an endpoint of the corresponding path. In addition each of the edges in $F_1\cup F_2$ joins distinct endpoints of paths in $Q^*$. Thus $F_0$ is disjoint from $F_1\cup F_2$. Thereafter, while constructing $F_1\cup F_2$ only edges between endpoints in $A(G)$ of paths of $Q^*$ were added. Since $q_1^*$ is even, $a(Q^*)$ must also be even, and so we have $|F_1\cup F_2|=\lceil \frac{1}{2}a(Q^*)\rceil =\frac{1}{2}a(Q^*)$, and therefore indeed
$$
|F| = |F_0| + |F_1\cup F_2| = |F_0|+ \frac{1}{2}a(Q^*) = \lceil \frac{1}{2}a(Q)\rceil.
$$

Now fix $\ell \in \{0,1,...,s\}$. Let $X_1,...,X_\ell$ be components in $S(G)$. For $1\leq i\leq \ell$, we have that $X_i$ is covered by a length $2$ path in $Q^*$ with its two endpoints in $B(G)$, and an additional path in $Q^*$ of length $0$ whose vertex lies in $B(G)$. Denote by $u(X_i)$ the unique vertex in $X_i\cap A(G)$, and by $w_1(X_i),w_2(X_i)$ the endpoint of the $2$-path spanned by $X_i$ in $Q^*$. We let $M_\ell'$ be a perfect matching containing $M$ on the set of all vertices in $B(G)\setminus \left( \bigcup _{i=1}^s X_i \right)$ that are endpoints of paths in $Q^*$ of positive length, such that if $v_1,v_2$ are the two endpoints in $B(G)$ of the same path in $Q^*$ then $\{v_1,v_2\}\in M_\ell$. Thereafter, we let $M_{\ell}=M_\ell' \cup \left( \bigcup _{i= 1}^{\ell} \left\{ \{ w_1(X_i),w_2(X_i) \} \right\} \right)$. Finally, let $B'(G)$ be $B(G)$ with the vertices that are internal vertices of paths in $Q^*$ removed. By Lemma \ref{lemma:4coreHmty}, $G\cup M_\ell$ contains a cycle $H_\ell'$ that covers $B'(G)\cup C(G)$, and contains all the edges of $M_\ell$  with probability $1-O(n^{-2})$. Replacing every edge $y_{2i-1}y_{2i}\in M\subseteq M_\ell$ of $H_\ell'$ with the path $P_{\ell}$ gives a cycle $H_{\ell}$ in $G\cup F$ that spans $V(G)\setminus \{ u(X_{\ell+1}),u(X_{\ell+2}),...,u(X_{s})\}$, and thus a cycle of length $n-(\ell-s)$ (see Figure \ref{fig:SG}). Hence $G\cup F$ spans a cycle of length $n-(s-\ell)$ with probability $1-O(n^{-2})$ for $0\leq \ell\leq s$. 
By the union bound, this completes the proof of \ref{lemma:mu-hamiltonian} and \ref{lemma:mu-longcycles}.

\begin{figure}[h]
\centering
\begin{tikzpicture}[vertex/.style={draw,circle,color=black,fill=black,inner sep=1,minimum width=4pt},scale=1]

	\draw[rounded corners=0.7cm] (0, 0) rectangle (9.7, -3) {};
	\node at (-0.5,-1.5) {$H_{\ell}$};
	
	\fill[white] (1,0) circle (0.3);
	\node[vertex] (b1) at (0.7,0) {};
	\node[vertex] (b2) at (1.3,0) {};
	\draw[very thick, dotted] (b1) to (b2);
	\node[vertex] (a1) at (1,0.8) {};
	\draw[thick] (b1) to (a1);
	\draw[thick] (b2) to (a1);
	\node at (0,0.3) {$w_1(X_1)$};
	\node at (1.5,-0.4) {$w_2(X_1)$};
	\node at (0.9,1.15) {$u(X_1)$};
	
	\fill[white] (2.1,0) circle (0.3);
	\node[vertex] (b3) at (1.8,0) {};
	\node[vertex] (b4) at (2.4,0) {};
	\draw[very thick, dotted] (b3) to (b4);
	\node[vertex] (a3) at (2.1,0.8) {};
	\draw[thick] (b3) to (a3);
	\draw[thick] (b4) to (a3);
	\node at (2.2,1.15) {$u(X_2)$};

	\coordinate (mid1) at (2.9,0.42) {};
	\coordinate (mid2) at (3.5,0.42) {};
	\draw[ultra thick, dotted] (mid1) to (mid2);

	\fill[white] (4.3,0) circle (0.3);
	\node[vertex] (b5) at (4,0) {};
	\node[vertex] (b6) at (4.6,0) {};
	\draw[very thick, dotted] (b5) to (b6);
	\node[vertex] (a5) at (4.3,0.8) {};
	\draw[thick] (b5) to (a5);
	\draw[thick] (b6) to (a5);
	\node at (4.2,1.15) {$u(X_{\ell})$};

	\fill[white] (3.2,0) circle (0.4);
	\draw[ultra thick, dotted] (b4) to (b5);

	\fill[white] (5.4,0) circle (0.3);
	\node[vertex] (b7) at (5.1,0) {};
	\node[vertex] (b8) at (5.7,0) {};
	\draw[thick] (b7) to (b8);
	\node[vertex] (a7) at (5.4,0.8) {};
	\node at (5.6,1.15) {$u(X_{\ell +1})$};
	
	\coordinate (mid3) at (6.2,0.42) {};
	\coordinate (mid4) at (6.8,0.42) {};
	\draw[ultra thick, dotted] (mid3) to (mid4);
	
	\fill[white] (7.6,0) circle (0.3);
	\node[vertex] (b9) at (7.3,0) {};
	\node[vertex] (b10) at (7.9,0) {};
	\draw[thick] (b9) to (b10);
	\node[vertex] (a9) at (7.6,0.8) {};
	\node at (7.4,1.15) {$u(X_{s-1})$};
	
	\fill[white] (6.5,0) circle (0.4);
	\draw[ultra thick, dotted] (b8) to (b9);

	\fill[white] (8.7,0) circle (0.3);
	\node[vertex] (b11) at (8.4,0) {};
	\node[vertex] (b12) at (9,0) {};
	\draw[thick] (b11) to (b12);
	\node[vertex] (a11) at (8.7,0.8) {};
	\node at (8.9,1.15) {$u(X_{s})$};

\end{tikzpicture}
\caption{\textit{With high probability, $G\cup F$ contains a cycle $H_{\ell}$ that passes through all the vertices except for $u(X_{\ell+1}),u(X_{\ell+2}),...,u(X_{s})$, and therefore a cycle of length $n-(\ell -s)$, for all $0\le \ell \le s$.}}\label{fig:SG}
\end{figure}

Finally, since we assumed that the properties listed in Lemma \ref{lemma:4coreProperties} all occur, and therefore $\cE (G)$ occurs, we have that for $k\in [3,\log\log n]$ either $G$ spans a cycle of length $k$ or $k\in K$. In the second case we have added to $F_0$ the edge $a_k^ka_1^k$. As a result $G\cup F$ contains the cycle $a_1^kb^ka_3^ka_4^k...a_k^ka_1^k$. This complete the proof of \ref{lemma:mu-shortcycles}.
\end{proof}

Equation \eqref{eq:ham}, Observation \ref{obs:mu} and Lemma \ref{lemma:mu} imply the following.
\begin{corollary}\label{corol:mu}
    Let $np\geq 20$ and $G\sim G(n,p)$. Then $\mu(G)=\mu'(G)$ with high probability.
\end{corollary}

\section{Proof of Theorem \ref{thm:main1}} \label{sec:proof1}
What follows is an adaptation of the proof from \cite{ANAS} for estimating the scaling limit of the length of the longest cycle in $G(n,p)$. For the rest of this section, we let $G\sim G(n,d/n)$ where $20\leq d\leq 0.4\log n$. To estimate $\mu(G)$, we first  define a sequence of random variables $\{\mu_{k}(G)\}_{k\geq 1}$. $\mu_{k}(G)$ will have the property that $|\mathbb{E}(\mu'(G)-\mu_k(G))|\leq (2ed)^{k}e^{-kd/4}\cdot n\leq 0.8^k\cdot n$. In addition, $\mathbb{E}(\mu_k(G))$ will depend on finitely many subgraph counts of $G$. The order of these subgraphs will depend on $k$. This enables us to calculate  $\mathbb{E}(\mu_k(G))$ exactly which we then use to approximate $\mathbb{E}(\mu'(G))$. The quantity $\mathbb{E}(\mu'(G))$ will end up being a good approximation of $\mu(G)$.

\subsection{Definition of \texorpdfstring{$\mu_k$}{Lg}} \label{sec:def-of-f}

Let $\mathcal{T}\left( G^{AB} \right)$ denote the set of connected components of $G^{AB}$. For $T\in \mathcal{T}\left( G^{AB} \right)$ let $\mathcal{P}_T$ denote the set of all disjoint path covers of $T$. For $P\in \mathcal{P}_T$ denote by $\phi (P)$ the number of endpoints of paths in $P$ that are members of $A(G)$, where, like in the definition of $\mu(G)$, a vertex of $A(G)$ that constitutes a path of length 0 in $P$ is counted twice towards $\phi (P)$. Finally, for $T\in \mathcal{T}\left( G^{AB} \right)$ define $\phi (T) = \min _{P\in \mathcal{P}_T}\left( \phi (P) \right)$, and for $v\in T$ define $\phi (v) = \phi (T)/|T|$. If $v\in C(G)$ we define $\phi (v)=0$. Evidently,
\begin{equation}\label{eq:a(g)-phi(v)}
a(G) = \sum _{T\in \mathcal{T}(G^{AB})}\phi (T) = \sum _{v\in V(G)} \phi (v).
\end{equation}

In the high probability event $\mu(G)=\mu'(G)$,  \eqref{eq:a(g)-phi(v)} implies that $\mu(G)$ can be expressed as the sum of $\phi(v)$ over $v\in V(G)$. By definition, if the value of $\phi(v)$ is non-zero then it can be determined by the component of $G^{AB}$ that contains $v$. Lemma \ref{lemma:4coreProperties} \ref{4core:sizeX_i} states that typically most of components of $G^{AB}$ are small. Thus, in the case that $v$ belongs to a small component of $G^{AB}$,  one may hope to be able to identify that component, and subsequently the value of $\phi(v)$, by just looking at the ball centered at $v$ of radius $k$, for sufficiently large $k$. Now considering the strong $4$-core of the subgraph of $G$ induced by the vertices within distance $k$ is no good, as typically, those vertices induce a tree, and every tree has an empty strong $4$-core. Instead of the $4$-core of that tree, we consider a similar set of vertices where we fix the vertices at distance $k$ (i.e. the boundary of the corresponding ball) to belong the ``strong $4$-core''. We will denote this set by $C(v,k)$. Based on the set $C(v,k)$ we will then define $\phi_k'(v)$. One should think of $\phi_k'(v)$ as a ``guess'' of $\phi(v)$ based on $N^{\leq k}_G(v)$.

Given a vertex $v\in V(G)$ and $k\geq 1$ we set $C(v,k)$ to be the maximal set $S\subseteq N_G^{< k}(v)$ with the property that every vertex in $S\cup N_{G}^{<k}(v)$ has at least $4$ neighbours in $S\cup N_{G}^{k}(v)$. We then let $B(v,k)$ be the set of vertices in  
$N_G^{< k}(v)\setminus C(v,k)$ that are adjacent to $C(v,k)$
and $A(v,k)= N_{G}^{<k}(S)\setminus (C(v,k) \cup B(v,k))$.

To define $\mu_k$, first define the function $\phi_k':V(G)\to [0,1]$ as follows. For $v\in V(G)$, given the sets $A(v,k), B(v,k)$ and $C(v,k)$ set $\phi_k'(v)=0$ if $v\in C(v,k)$ . Else let $T^{AB}(v,k)$ be the component containing $v$ in the subgraph of $G[N_{G}^{<k}(v)]$ induced by $A(v,k) \cup B(v,k)$ and set $\phi_k'(v)= \frac{\phi(T^{AB}(v,k))}{|T^{AB}(v,k)|}$. 

Thereafter we define the function $\phi_k:V(G)\to [0,1]$ by $\phi_k(v)=\phi_k'(v)$ if  $|N_G^{\leq k}(v)|\leq 2d^ke^{kd}$ and $\phi_k(v)=0$ otherwise for $v\in V(G)$. Finally, we let
\begin{equation}\label{eq:localversion}
    \mu_k(G)= \frac{1}{2}\cdot \sum_{v \in V(G)} \phi_k(v).
\end{equation}  
Note that $\phi_k$ is a truncated version of $\phi_k'$. In particular, the two functions are not equal only on vertices $v$ for which the size of $N^{\leq k}_G(v)$ is significant larger that its expected value. Thus one may obtain $\mu_k(G)$ from $a(G)/2$ by replacing $\phi(v)$ in  \eqref{eq:a(g)-phi(v)} with the corresponding truncated guesses $\phi_k(v)$, that depend on $N_G^{\leq k}(v)$, for $v\in V(G)$.  Recall that $\lceil a(G)/2 \rceil =\mu(G)$ with high probability whenever $G\sim G(n,p)$, $np\geq 20$.  Hence,
\begin{align}\label{eq:localdiff}
|\mu'(G)-\mu_k(G)|& = \bigg|\bigg\lceil\frac{1}{2}\cdot \sum_{v \in V(G)} \phi(v) \bigg\rceil -\frac{1}{2}\cdot \sum_{v \in V(G)} \phi_k(v)  \bigg|
\leq 1+ \bigg|\frac{1}{2}\cdot \sum_{v \in V(G)} (\phi(v)  -\phi_k(v))  \bigg| \nonumber
\\&\leq 1+ \frac{1}{2}\sum_{v\in V(G)} \mathbb{I}(\phi_k(v)\neq \phi(v)) \leq 1+ \sum_{v\in V(G)} \mathbb{I}(\phi_k(v)\neq \phi(v)) \nonumber
\\&\leq 1+ |\{v\in V(G): \phi_k'(v)\neq \phi(v)\}|+|\{v\in V(G):|N_G^{\leq k}(v)|>2d^ke^{kd}\}|.
\end{align}  
The second inequality above follows from the fact that both values $\phi(v),\phi_k(v)$ belong to $[0,1]$ for $v\in V(G)$.

For $v\in V(G)$ denote by $T^{AB}(v)$ the component of $G^{AB}$ containing $v$. In particular if $v\in C(G)$ then $T^{AB}(v)=\emptyset$.
In Lemma \ref{lem:phi'approx} we show that the guess $\phi_k'(v)$ equals $\phi(v)$ for every vertex $v$ such that $|T^{AB}(v)|\leq k-1$ (recall that  Lemma \ref{lemma:4coreProperties} \ref{4core:sizeX_i} states that most of the vertices satisfy this condition). In Lemma \ref{lem:boundapprox} we combine \eqref{eq:localdiff}, Lemma \ref{lem:phi'approx} and Markov's inequality to show that $|\mathbb{E}(\mu'(G))-\mathbb{E}(\mu_k(G))|$ decays exponentially with $k$.

\begin{lemma}\label{lem:phi'approx}
 If $|T^{AB}(v)|\leq k-1$ then $ \phi_k'(v) = \phi(v)$.
\end{lemma}
\begin{proof}
Let $v\in V(G)$ be such that $|T^{AB}(v)|\leq k-1$. Set $D^*=C(G)\cap N_G^{k}(v)$. Then by definition of $C(G)$, the set $C(G)\cap N_G^{<k}(v)$ is the maximal subset $S$ of $N_G^{<k}(v)$ with the property that every vertex in $S\cup N_G(S)$ has at least $4$ neighbours in $S\cup D^*$. As $D^*\subseteq N_G^{k}(v)$, one has $C(G)\cap N_G^{<k}(v)\subseteq C(v,k)$. It follows that if $v\in C(G)$ then $v\in C(v,k)$ and $\phi(v)=0=\phi_k(v)$. 

Now assume that $v\notin C(G)$, equivalently that $T^{AB}(v)\neq \emptyset$ and let $D=N_G(T^{AB}(v))$. As $T^{AB}(v)$ is connected, contains $v$ and has size at most $k-1$ we have that $T^{AB}(v)\subseteq N_G^{\leq k-2}(v)$ and therefore $D\subseteq N_G^{< k}(v)$. In addition, as $D\subseteq C(G)$ we have that $D\subseteq C(v,k)$. By the definitions of $C(G),C(v,k)$ we have that both of the sets $C(G)\cap T^{AB}(v)$, $C(v,k)\cap T^{AB}(v)$ are the maximal subset $S$ of $T^{AB}(v)$ with the property that every vertex in $S\cup N_{G[T^{AB}(v)]}(S)$ has at least $4$ neighbours in $S\cup D$. Thus, we have equalities
\begin{eqnarray*}
& & T^{AB}(v)\cap C(G) = T^{AB}(v)\cap C(v,k),\\
& & N_G(D)\cap T^{AB}(v) = T^{AB}(v)\cap B(G)= T^{AB}(v)\cap B(v,k),\\
& & T^{AB}(v)\cap A(G)= T^{AB}(v)\cap A(v,k).
\end{eqnarray*}
It now follows that $\phi(v)=\phi_k(v)$.
\end{proof}

\begin{lemma}\label{lem:boundapprox}
For $k\geq 1$, $|\mathbb{E}(\mu'(G))-\mathbb{E}(\mu_k(G))|\leq (2ed)^{k}e^{-kd/4} \cdot n$.
\end{lemma}
\begin{proof}
Note that for $k\in \mathbb{N}^+$ and $v\in [n]$ we have $\mathbb{E}(|N_G^k(v)|)\leq d^k$. Therefore, $\mathbb{E}(|N_G^{\leq k}(v)|)\leq 2d^k$, and by Markov's inequality we get
$$\pr \left( |N_G^{\leq k}(v)|)\geq 2d^ke^{dk} \right) \le e^{-dk}.$$
In extension,
$$\mathbb{E}(|\{v\in V(G):|N_G^{\leq k}(v)|>2d^ke^{kd}\}|)\leq e^{-dk} n.$$
To bound the expected number of $|\{v\in V(G): \phi_k'(v)\neq \phi(v)\}|$ recall that Lemma \ref{lem:phi'approx} implies that if $\phi_k'(v)\neq \phi(v)$ then $|T^{AB}(v)|\geq k$. In addition, $T^{AB}(v)$ has at least  $|T^{AB}(v)|/4$ vertices in $A(G)$. Thus, Lemma  \ref{lemma:4coreProperties} \ref{4core:sizeX_i} implies that 
$$\mathbb{E}(|\{v\in V(G): \phi_k'(v)\neq \phi(v)\}|)\leq  \sum_{i\geq k}
\frac{(2ed)^{i}e^{-(i/4)d}}{15(i/4)d}\cdot n
\leq 4\cdot \frac{(2ed)^{k}e^{-kd/4}}{15kd/4}\cdot n.$$
 \eqref{eq:localdiff}, and the above imply that
$$|\mathbb{E}(\mu'(G))-\mathbb{E}(\mu_k(G))|\leq 1+ \sum_{i\geq k} 4\cdot \frac{(2ed)^{k}e^{-kd/4}}{15kd/4}\cdot n +e^{-dk}\cdot n\leq (2ed)^{k}e^{-kd/4}\cdot n.$$
\end{proof}

\subsection{The scaling limits of the approximations}\label{sec:limits}
For $\ell,k\geq 1$ we let $\mathcal{H}_{k,\ell}$ be the set of 
pairs $(H,o_H)$ where $H$ is a connected graph, $o_H$ is a distinguished vertex of $H$, that is considered to be the root, every vertex in $V(H)$ is within distance at most $k$ from $o_H$, there are at most $2d^ke^{dk}+1$ vertices in $H$ and $\ell$ of the vertices in $H$ are at distance at most $k-1$ from $o_H$. For $(H,o_H) \in \mathcal{H}_{k,\ell}$ let $X_{(H,o_H)}(G)$ be the number of copies of $(H,o_H)$ in $G$. Also let $\phi(H,o_H)$ be equal to the value of $\phi_k(v)$ in the event $(G[N^{\leq k}(v)],v)=(H,o_H)$.
Then,
\begin{align*}
   \mu_k(G)&=\frac{1}{2}\cdot \sum_{v \in V(G)} \phi_k(v)
    =\frac{1}{2}\cdot  \sum_{v \in V(G)}  \left(
    \sum_{\ell \geq 1} \sum_{(H,o_H) \in \mathcal{H}_{k,\ell }} \phi(H,o_H) \mathbb{I}\left((G[N^{\leq k}(v)],v)=(H,o_H)\right) \right)
    \\  & =\frac{1}{2}  \cdot
    \sum_{\ell \geq 1} \sum_{(H,o_H) \in \mathcal{H}_{k,\ell }} \phi(H,o_H) \left(  \sum_{v \in V(G)} \mathbb{I}\left((G[N^{\leq k}(v)],v)=(H,o_H)\right) \right)
 \\&   =\frac{1}{2}\cdot \sum_{\ell \geq 1} \sum_{(H,o_H) \in \mathcal{H}_{k,\ell }} \phi(H,o_H) X_{(H,o_H)}(G).
\end{align*} 
For $k\geq 1$ we let 
$$\rho_{d,k}= \sum_{\ell \geq 1} \sum_{\substack{(H,o_H) \in \mathcal{H}_{k,\ell }:\\H \text{ is a tree }}}  \frac{\phi(H,o_H) \cdot d^{|V(H)|-1} \cdot e^{-d\ell}}{2\cdot aut(H,o_H)}.$$
Here by $aut(H,o_H)$ we denote the number of automorphisms of $H$ that map $o_H$ to $o_H$. Then, 
\begin{align}\label{def:prerho}
 \mathbb{E}\bigg( \frac{\mu_k(G)}{n} \bigg) &=\sum_{\ell\geq 1} \sum_{(H,o_H) \in \mathcal{H}_{k,\ell}} \frac{\phi(H,o_H) \mathbb{E}(X_{(H,o_H)}(G))}{2n} \nonumber
\\&= \sum_{\ell\geq 1}\sum_{(H,o_H) \in \mathcal{H}_{k,\ell}} \frac{ \phi(H,o_H)  \cdot\binom{n}{|V(H)|} \cdot|V(H)|!\cdot p^{|E(H)|} \cdot(1-p)^{\ell\cdot (n-|V(H)|)+ \binom{|V(H)|}{2}-|E(H)|}}{2\cdot aut(H,o_H) \cdot n} \nonumber
\\&= \sum_{\ell\geq 1} \sum_{(H,o_H) \in \mathcal{H}_{k,\ell}} \frac{ \phi(H,o_H)  \cdot n^{|V(H)|}\cdot p^{|E(H)|}\cdot e^{-d \ell }}{2\cdot aut(H,o_H) \cdot n}+O(n^{-0.5}).
\end{align}
At the last equality we used that $\mathbb{E} \bfrac{\mu_k(G)}{n}\in [0,1]$, and that 
\begin{align*}
&\binom{n}{|V(H)|} \cdot|V(H)|! \cdot (1-p)^{\ell\cdot (n-|V(H)|)+ \binom{|V(H)|}{2}-|E(H)|}= \prod_{i=0}^{|V(H)|-1}\left(1-\frac{i}{n}\right)\cdot n^{|V(H)|} \cdot e^{-(p+O(p^2))(\ell n +O(1))}
\\&   =\left(1-O\bfrac{1}{n}\right)\cdot n^{|V(H)|} \cdot e^{-np\ell+O(np^2+p)}=n^{|V(H)|} \cdot e^{-d\ell} \cdot \left(1+O(n^{-0.5})\right).
\end{align*}
For $(H,o_H)\in \mathcal{H}_{k,\ell}$ that is not a tree, we have that,
\begin{equation}\label{eq:nontrees}
 \frac{ \phi(H,o_H) \cdot n^{|V(H)|}\cdot p^{|E(H)|} \cdot e^{-d\ell}}{2\cdot aut(H,o_H) \cdot n}\leq \frac{n^{|V(H)|}\cdot p^{|E(H)|}}{n}\leq \frac{(np)^{|V(H)|}}{n} = \frac{d^{|V(H)|}}{n}
=O(n^{-0.5}).    
\end{equation}
At the first inequality above we used that $\phi(H,o_H)\in[0,1]$  and at the equality we used that $|V(H)|=O(1)$, and $d=O(\log n)$. 

Finally, as the sum in \eqref{def:prerho} is taken over finitely many pairs $(H,o_H)$ (hence it involves only finitely many pairs where $H$ is not a tree), \eqref{def:prerho} and \eqref{eq:nontrees} give that,
\begin{align}\label{def:rho}
 \mathbb{E}\bigg( \frac{\mu_k(G)}{n} \bigg) =  \sum_{\ell\geq 1} \sum_{\substack{(H,o_H) \in \mathcal{H}_{k,\ell}:\\H \text{ is a tree }}} \frac{ \phi(H,o_H)  \cdot n^{|V(H)|}\cdot p^{|V(H)-1|}e^{-d \ell }}{2\cdot aut(H,o_H) \cdot n}+O(n^{-0.5}) =\rho_{d,k}+O(n^{-0.5}). 
\end{align}

\subsection{The scaling limit of  \texorpdfstring{$\mu(G)$}{Lg}}  \label{subsec:scalinglimit}

We now define the function $f:[0,\infty)\to [0,1]$ by 
\begin{align*}
    f(d)=\begin{cases}
         \rho_{d,1}+\sum_{i\geq 1}(\rho_{d,i+1}-\rho_{d,i}) &\text{ for }d\geq 20;
     \\ \rho_{20,1}+\sum_{i\geq 1}(\rho_{20,i+1}-\rho_{20,i}) &\text{ for }0\leq d<20.
    \end{cases}
\end{align*}
\begin{lemma}\label{lem:thm1:part2}
 With high probability, for all $k\geq 1$
\begin{equation}\label{eq:lem:thm1:part2}
|\mu(G)-f(d)\cdot n|\leq 3 n\cdot 0.8^k +O(n^{0.5}).
\end{equation} 
\end{lemma}
In the proof of Lemma \ref{lem:thm1:part2} we use Lemma \ref{lem:muconcentration}. Its proof is located in Appendix \ref{app:sec:lemma:4coreProperties}.
\begin{lemma}\label{lem:muconcentration}
 $ |\mu'(G)-\mathbb{E}(\mu'(G))|\leq n^{0.51} $
with high probability.   
\end{lemma}

\begin{proof}[Proof of Lemma \ref{lem:thm1:part2}]
Lemma \ref{lem:boundapprox} and \eqref{def:rho} imply that for $1\leq k_1<k_2$,
\begin{align*}
n|\rho_{d,k_1}-\rho_{d,k_2}| &\leq |n\rho_{d,k_1}-\mathbb{E}(\mu_{k_1}(G))|+|\mathbb{E}(\mu_{k_1}(G))-\mathbb{E}(\mu'(G))|\nonumber
\\&+|\mathbb{E}(\mu'(G))-\mathbb{E}(\mu_{k_2}(G))|+|\mathbb{E}(\mu_{k_2}(G))-\rho_{d,k_2}|\nonumber
\\&\leq  O(n^{0.5})+ ((2ed)^{k_1}e^{-k_1d/4} \cdot n+ O(n^{0.5}))+((2ed)^{k_2}e^{-k_2d/4} \cdot n+O(n^{0.5}))+O(n^{0.5})
\\&\leq 2\cdot (2ed)^{k_1}e^{-k_1d/4} \cdot n+O(n^{0.5}).    
\end{align*}
Thus, for $1\leq k_1<k_2$,
\begin{align}\label{eq:exp_to_rho}
|\rho_{d,k_1}-\rho_{d,k_2}| & \leq  2\cdot (2ed)^{k_1}e^{-k_1d/4}.    
\end{align}

It follows that the sum in the definition of $f(d)$, and in extension $f(d)$, is well defined for all $d\geq 0$.
In addition, as $\rho_{d,k}$ is continuous in $d$, we have that $f$ is continuous. Also note that for $k\geq 1$ and $d\geq 20$,
$$ f(d)= \rho_{d,1}+\sum_{i\geq 1}(\rho_{d,i+1}-\rho_{d,i})= \rho_{d,k}+\sum_{i\geq k+1}(\rho_{d,i+1}-\rho_{d,i}).$$
Thus, we have that  for all $k\geq 1$,
\begin{align}\label{eq:approxnubyfd}
    | \mathbb{E}(\mu'(G))-f(d) \cdot n|
   &=\bigg | \mathbb{E}(\mu'(G))-n\cdot \rho_{d,k}-n\cdot \sum_{i\geq k+1}(\rho_{d,i+1}-\rho_{d,i})\bigg| \nonumber
    \\    &\leq |\mathbb{E}(\mu'(G))- \mathbb{E}(\mu_k(G))|+|\mathbb{E}(\mu_k(G))- n\cdot \rho_{d,k}|+n\cdot \bigg|\sum_{i\geq k}(\rho_{d,i+1}-\rho_{d,i}) \bigg| \nonumber
    \\&\leq  n\cdot  (2ed)^{k}e^{-kd/4} +O(n^{0.5}) + n\cdot \sum_{i\geq k}  2\cdot (2ed)^{k}e^{-kd/4} \nonumber
    \\& \leq 3n\cdot (2ed)^{k}e^{-kd/4}  +O(n^{0.5}).
\end{align} 
At the penultimate inequality we used Lemma \ref{lem:boundapprox}, \eqref{def:rho}, and \eqref{eq:exp_to_rho}. \eqref{eq:approxnubyfd} and Lemma \ref{lem:muconcentration} imply that, for all $k\geq 1$, with high probability
$$|\mu'(G)-f(d)\cdot n|\leq  3n\cdot (2ed)^{k}e^{-kd/4}  +O(n^{0.5}). \leq 3n\cdot 0.8^k+O(n^{0.5}).$$
\eqref{eq:lem:thm1:part2} follows from the inequality above and Corollary \ref{corol:mu}.
\end{proof}

\subsection{Approximating the scaling limit of \texorpdfstring{$\mu(G)$}{Lg}}\label{subsec:approx}

We now proceed with approximating $f(d)$ up to accuracy $O\left( d^{15}e^{-4d} \right)$. In order to approximate $f(d)$ we use Lemma \ref{lemma:proof3comb}, which is stated shortly, and proved at the end of this section. For a graph $G$ and $T$ a connected component in $G^{AB}$ denote $A(T) \coloneqq A(G)\cap V(T)$ and $B(T)\coloneqq B(G)\cap V(T)$, and denote by $n_i(T)$ the number of vertices in $A(T)$ that have degree $i$.

A $3$-prespider of a graph $F$ is a tree-subgraph of $F$ whose edge set consists of the edges incident to $3$ vertices of degree at most $2$ in $F$ with a common neighbour. There are 4 non-isomorphic $3$-prespiders (see Figure \ref{fig:prespiders}). 

We also let  $s_3'(G)$ and $s_3'(T)$ be the number of 3-prespiders spanned by $G$, $T$ respectively. Finally, let $a(T)$ be the number of endpoints in $A(T)$ of paths in a cover $Q$ with $a(Q) =a(G)$ (observe that $a(T)$ does not depend on the choice of $Q$, as long as $a(Q)=a(G)$).

\begin{figure}[h]
\centering
\begin{tikzpicture}[vertex/.style={draw,circle,color=black,fill=black,inner sep=1,minimum width=4pt},scale=1]

	\coordinate (l1) at (-1.8,0.15) {};
	\coordinate (l2) at (-1.3,0.15) {};
	\coordinate (l3) at (7,0.15) {};
	\coordinate (l4) at (7.5,0.15) {};
	\draw[dashed] (l1) to (l2);
	\draw[] (l2) to (l3);
	\draw[dashed] (l3) to (l4);
	\node at (-2.2,0.5) {$A(G)$};
	%\node at (-2.2,-0.2) {$B(G)$};	

	\node[vertex] (b11) at (-0.5,-0.3) {};
	\node[vertex] (a11) at (-0.5,0.6) {};
	\node[vertex] (a12) at (0,0.6) {};
	\node[vertex] (a13) at (-1,0.6) {};
	\draw[thick] (b11) to (a11);
	\draw[thick] (b11) to (a12);
	\draw[thick] (b11) to (a13);
	\node at (-0.5,-0.9) {$\frac{d^3}{3!}e^{-3d}n$};

	\node[vertex] (b21) at (1.5,-0.3) {};
	\node[vertex] (b22) at (0.7,-0.3) {};
	\node[vertex] (a21) at (1.5,0.6) {};
	\node[vertex] (a22) at (2,0.6) {};
	\node[vertex] (a23) at (1,0.6) {};
	\draw[thick] (b21) to (a21);
	\draw[thick] (b21) to (a22);
	\draw[thick] (b21) to (a23);
	\draw[thick] (b22) to (a23);
	\node at (1.5,-0.9) {$\frac{3d^4}{3!}e^{-3d}n$};

	\node[vertex] (b31) at (3.5,-0.3) {};
	\node[vertex] (b32) at (2.7,-0.3) {};
	\node[vertex] (b33) at (4.3,-0.3) {};
	\node[vertex] (a31) at (3.5,0.6) {};
	\node[vertex] (a32) at (4,0.6) {};
	\node[vertex] (a33) at (3,0.6) {};
	\draw[thick] (b31) to (a31);
	\draw[thick] (b31) to (a32);
	\draw[thick] (b31) to (a33);
	\draw[thick] (b32) to (a33);
	\draw[thick] (b33) to (a32);
	\node at (3.5,-0.9) {$\frac{3d^5}{3!}e^{-3d}n$};

	\node[vertex] (b41) at (5.5,-0.3) {};
	\node[vertex] (b42) at (4.7,-0.3) {};
	\node[vertex] (b43) at (6.7,-0.3) {};
	\node[vertex] (b44) at (6.1,-0.3) {};
	\node[vertex] (a41) at (5.5,0.6) {};
	\node[vertex] (a42) at (6,0.6) {};
	\node[vertex] (a43) at (5,0.6) {};
	\draw[thick] (b41) to (a41);
	\draw[thick] (b41) to (a42);
	\draw[thick] (b41) to (a43);
	\draw[thick] (b42) to (a43);
	\draw[thick] (b43) to (a42);
	\draw[thick] (b44) to (a41);
	\node at (5.5,-0.9) {$\frac{d^6}{3!}e^{-3d}n$};

\end{tikzpicture}
\caption{\textit{The $4$ non-isomorphic $3$-prespiders and the respective leading terms of the expected number of their appearances. Here, the vertices above the line belong to $A(G)$, and have no neighbours in $G$ outside the $3$-prespider, while the vertices below the line may belong to either $B(G)$ or $A(G)$ and may have more neighbours.}}\label{fig:prespiders}
\end{figure}

\begin{lemma} \label{lemma:proof3comb}
If $T$ is a tree connected component $G^{AB}$, with $|A(T)|\leq 3$, then $a (T) = 2n_0(T) + n_1(T)+s_3'(T)$.
\end{lemma}

The main lemma of this subsection is the following.
\begin{lemma}\label{lem:thm1:part3}
Let $d=d(n)$ be such that $20\leq d\leq 0.4\log n$. Then, 
$$f(d) = \frac{1}{2}de^{-d} + e^{-d}+\left(\frac{1}{12}d^6 + \frac{1}{4}d^5 +\frac{1}{4}d^4 +\frac{1}{12}d^3 \right) \cdot e^{-3d} + O_d\left( d^{15}e^{-4d} \right).$$ 
\end{lemma}
\begin{proof}
    
Let $\cT'\subseteq \mathcal{T}\left( G^{AB} \right)$ be the set of tree components  $T$ of $G^{AB}$ with at most $3$ vertices in $A$. Let $\cP$ be a set of vertex disjoint paths that cover $G^{AB}$ such that (i) each   component $T\in \cT'$ is covered by paths with $a(T)$ endpoints in $A(T)$ in total and
(ii) the rest of the components are covered by paths arbitrarily. Let $a(\cP)$ be the total number of endpoints of paths in $\cP$ that belong to $A(G)$. By the definition of $\mu'(G)$ we have that 
$2\cdot \mathbb{E}(\mu'(G))\leq \mathbb{E}(a(\cP))+1$. In addition note that for every component $C$ of $G^{AB}$ with $C\cap A(G)=i$, the number of endpoints of paths in $\cP$ that belong to $C\cap A(G)$ is at most $2i$. Thus, 
\begin{align}\label{eq:mubound1}
\mathbb{E}(2\mu'(G))&- \mathbb{E}(2n_0(G)+n_1(G)+s_3'(G))  \leq \mathbb{E}(a(\cP))+1- \mathbb{E}\left(\sum_{T\in \cT':|A(T)|\leq 3} 2n_0(T)+n_1(T)+s_3'(T)\right) \nonumber
\\&\leq 1+ 30\cdot d^{12}+\sum_{i\geq 4}\mathbb{E}(2i\cdot Y_i)
\leq 1+ 30\cdot d^{12}+\sum_{i\geq 4}i^3\mathbb{E}(Y_i),
\end{align}

where $Y_i$ is the number of components of $G^{AB}$ with $i$ vertices in $A$. The $30d^{12}$ term in \eqref{eq:mubound1} is a crude upper bound on the expected number of endpoints of paths in $\cP$ that are spanned by cyclic components of $G^{AB}$ with at most $3$ vertices in $A(G)$, hence with at most $12$ vertices in total. For the last inequality  in \eqref{eq:mubound1} we used that Lemma \ref{lemma:proof3comb} implies that 
$\mathbb{E}\left(\sum_{T\in \cT':|A(T)|\leq 3} 2n_0(T)+n_1(T)+s_3'(T)\right)
= \mathbb{E}\left(\sum_{T\in \cT':|A(T)|\leq 3}a(T)\right)$.

On the other hand, Lemma \ref{lemma:proof3comb} implies that $2\mu'(G)- (2n_0(G)+n_1(G)+s_3'(G))$ is lower bounded by minus the sum of 
$2n_0(T)+n_1(T)+s_3'(T)$ over components $T\in \mathcal{T}\left( G^{AB} \right)$ such that either $|A(T)|\geq 4$ or $A(T)\leq 3$ (hence $T$ spans at most 12 vertices) and $T$ spans a cycle. $X_i$ is the number of components $T$ with $|A(T)|=i$. For $i\geq 4$ each such component spans at most $i$ vertices of degree $1$, $0$ vertices of degree $0$ and at most $\binom{i}{3}$ $3$-prespiders (each $3$-prespider is uniquely determined by the $3$ vertices with the common neighbour that it spans, each such vertex belongs to $A(G)$). Thus,
\begin{equation}\label{eq:mubound2}
    \mathbb{E}(2\mu'(G))- \mathbb{E}(2n_0(G)+n_1(G)+s_3'(G)) \geq -
\mathbb{E}\bigg(\sum_{i\geq 4}\bigg(i+\binom{i}{3}\bigg)\cdot Y_i\bigg)-30d^{12} \geq -1-30d^{12}-\sum_{i\geq 4}i^3\mathbb{E}(Y_i).
\end{equation}
Lemma  \ref{lemma:4coreProperties} \ref{4core:sizeX_i} implies that $\mathbb{E}(Y_i) \le \max \left\{ \frac{(2ed)^{4i}e^{-id}}{15id}\cdot n, n^{-3}\right\}$ for every $1\le i\le n$. Thus, \eqref{eq:mubound1} and \eqref{eq:mubound2} give that 
\begin{align*}
|2\mathbb{E}(\mu'(G))- \mathbb{E}(2n_0(G)+n_1(G)+s_3'(G))|
&\leq \sum_{i\geq 4}i^3\mathbb{E}(Y_i) +1+ 30d^{12} 
\leq \sum_{i\geq 4} i^3 \frac{(2ed)^{4i}e^{-id}}{15id} n+1+30d^{12}
\\&\leq \frac{16 (2e)^4d^{15}e^{-4d}}{15} \cdot n\cdot  
 \sum_{i\geq 4} \frac{i^2}{16} \left((2ed)^4 e^{-d}\right)^{i-4}+1+30d^{12}
\\&=O_d(d^{15}e^{-4d})\cdot n.
\end{align*}

A $3$-prespider corresponds to a connected subgraph of a $3$-spider that spans a vertex of degree $3$, thus there are exactly $4$ non-isomorphic $3$-prespiders (see Figure \ref{fig:prespiders}). By calculating the expected number of appearances of each such 3-prespider in $G(n,p)$, it follows that 
$$\mathbb{E}(2n_0(G)+n_1(G)+s_3'(G))= \bigg(2e^{-d}+de^{-d}+ \bigg(\frac{d^3}{3!}+\frac{3d^4}{3!}+\frac{3d^5}{3!}+\frac{d^6}{3!} \bigg) e^{-3d}\bigg)n+ O(n^{0.5}).$$
The above and \eqref{eq:approxnubyfd} (taking $k=20$) imply that 
$$f(d)=e^{-d}+\frac{1}{2}de^{-d}+ \bigg(\frac{d^3}{12}+\frac{d^4}{4}+\frac{d^5}{4}+\frac{d^6}{12} \bigg) e^{-3d}+O_d(d^{15}e^{-4d}).$$
\end{proof}

\begin{proof}[Proof of Lemma  \ref{lemma:proof3comb}] 
First, assume that $|A(T)|=1$, and let $v$ be the unique vertex in $A(T)$. Then taking $\min \{ d(v),2 \}$ arbitrary edges adjacent to $v$, and covering all uncovered vertices with paths of length 0, results in a disjoint path cover in which $v$ is an endpoint if and only if $d(v)<2$. Since $A(T)=\{v\}$, this means that in particular $a (T) = 2n_0(T) + n_1(T)= 2n_0(T) + n_1(T)+s_3'(T)$.

Now assume that $|A(T)|=2$ and let $A(T) = \{u,v\}$. Then there exists a $u$-$v$ path $P$ spanned by $T$. For $w\in \{u,v\}$, if $w$ has degree at least $2$ then extend $P$ by adding to it a second edge adjacent to $w$. Since $T$ is acyclic this gives a path $P'$ with $n_1(T)$ endpoints in $A$. As $n_1(T)$ is a lower bound on $a (T)$ we have that $a(T) = n_1(T)= 2n_0(T) + n_1(T)+s_3'(T)$.

Finally, consider the case $|A(T)|=3$ and let $A(T)=\{w_1,w_2,w_3\}$. Let $P_1$ be a $w_1$-$w_2$ path spanned by $T$ and $P_2$ be a shortest path from $w_3$ to a vertex $w\in V(P_1)$. As $T$ is acyclic $P_2$ is well defined. If $w\in \{w_1,w_2\}$ then concatenating the two paths gives a path $P'$ that covers $A(T)$ and has endpoints in $A(T)$ (see Figure \ref{fig:lem5.4} case (a)). Similar to the case $|A(T)|=2$ this implies that  $a(T) = n_1(T)= 2n_0(T) + n_1(T)+s_3(T)$. Otherwise, $w\notin \{w_1,w_2\}$ and there exists a path $R_i$ from $w_i$ to $w$ spanned by  $T$ for $i=1,2,3$. First, assume that there exists $u\in A(T)$ such either $u$ is not adjacent to $w$ or $d_T(u)\geq 3$. Then $u$ has at least $\min\{d(u),2\}$ neighbours in $V(T)\setminus (A(T)\cup \{w\})$. Without loss of generality assume that $u=w_1$. Then by starting with the paths $R=w_2R_2wR_3w_3$ and $w_1$, augmenting $R$ as in the case $|A(T)|=2$ and connecting $w_1$ to $\min\{d(u),2\}$ neighbours of it not in $V(T)\setminus (A(T)\cup \{w\})$ once again we have that $a(T) = n_1(T)= 2n_0(T) + n_1(T)+s_3'(T)$ (see Figure \ref{fig:lem5.4} case (b)).

This leaves the case where 
 $A(T)=\{w_1,w_2,w_3\}$, the vertices $w_1,w_2,w_3$ have degree at most $2$ and a common neighbour $w$, which means that $T$ spans a $3$-prespider of $G$. In any disjoint path covering $\cP$ of $T$ the vertex $w$ is covered by a single path and therefore at least $1$ of the edges $ww_1,ww_2,ww_3$ is not spanned by some path. Therefore $\cP$ spans at most $d(w_1)+d(w_2)+d(w_3)-1$ edges incident to $A(T)$. As $T$ is acyclic and  $w_1,w_2,w_3$ are adjacent to $w$ we have that $A(T)$ does not span an edge of $T$. Thus, $\cP$ has at least $6-(d(w_1)+d(w_2)+d(w_3)-1)=n_1(T)+1$ endpoints in $A(T)$. On the other hand, by greedily extending the paths $w_1\, w\, w_2$ and $w_3$ into vertex disjoint paths we may construct a disjoint path covering of $T$ with $6-(d(w_1)+d(w_2)+d(w_3)-1)=n_1(T)+1$ endpoints in $A(T)$. Hence, $a(T)=n_1(T)+1=2n_0(T)+n_1(T)+s_3'(T)$ (see Figure \ref{fig:lem5.4} case (c)).

\begin{figure}[h]
\centering
\begin{tikzpicture}[vertex/.style={draw,circle,color=black,fill=black,inner sep=1,minimum width=4pt},scale=1]

	\coordinate (l1) at (-2.1,0.15) {};
	\coordinate (l2) at (-1.6,0.15) {};
	\coordinate (l3) at (7,0.15) {};
	\coordinate (l4) at (7.5,0.15) {};
	\draw[dashed] (l1) to (l2);
	\draw[] (l2) to (l3);
	\draw[dashed] (l3) to (l4);
	\node at (-2.5,0.5) {$A(G)$};
	\node at (-2.5,-0.2) {$B(G)$};

	\node[vertex] (b11) at (-1.1,-0.3) {};
	\node[vertex] (b12) at (-0.3,-0.3) {};
	\node[vertex] (b13) at (0.1,-0.3) {};
	\node[vertex] (b14) at (0.5,-0.3) {};
	
	\node[vertex] (a12) at (-0.3,0.6) {};
	\node[vertex] (a13) at (0.5,0.6) {};
	\node[vertex] (a11) at (-1.1,0.6) {};
	\node at (-0.1,0.95) {$w_2=w$};
	\node at (0.9,0.6) {$w_3$};
	\node at (-1.5,0.6) {$w_1$};
	
	\draw[thick] (b11) to (a11);
	\draw[thick] (b12) to (a12);
	\draw[thick] (b11) to (b12);
	\draw[thick] (a13) to (b13);
	\draw[thick] (a13) to (b14);
	\draw[thick] (a12) to (b13);
	
	\node at (-0.3,-1.2) {case (a)};

	\node[vertex] (b31) at (1.8,-0.3) {};
	\node[vertex] (b32) at (2.6,-0.3) {};
	\node[vertex] (b33) at (3,-0.3) {};
	\node[vertex] (b34) at (3.4,-0.3) {};
	
	\node[vertex] (a31) at (1.8,0.6) {};
	\node[vertex] (a32) at (2.6,0.6) {};
	\node[vertex] (a33) at (3.4,0.6) {};
	\node at (1.8,0.9) {$w_1$};
	\node at (2.6,0.9) {$w_2$};
	\node at (3.4,0.9) {$w_3$};
	\node at (2.6,-0.6) {$w$};
	
	\draw[thick] (b31) to (a31);
	\draw[thick] (b32) to (a32);
	\draw[thick] (b33) to (a33);
	\draw[thick] (b34) to (a33);
	\draw[thick, densely dotted] (b31) to (b32);
	\draw[thick] (b33) to (b32);
	
	\node at (2.6,-1.2) {case (b)};

	\node[vertex] (b41) at (4.7,-0.3) {};
	\node[vertex] (b42) at (5.5,-0.3) {};
	\node[vertex] (b43) at (6.3,-0.3) {};
	
	\node[vertex] (a41) at (4.7,0.6) {};
	\node[vertex] (a42) at (5.5,0.6) {};
	\node[vertex] (a43) at (6.3,0.6) {};
	\node at (4.7,0.9) {$w_1$};
	\node at (5.5,0.9) {$w_2$};
	\node at (6.3,0.9) {$w_3$};
	\node at (5.5,-0.6) {$w$};
	
	\draw[thick] (b42) to (a41);
	\draw[thick] (b42) to (a42);
	\draw[thick, densely dotted] (b42) to (a43);
	\draw[thick] (b41) to (a41);
	\draw[thick] (b43) to (a43);
	
	\node at (5.5,-1.2) {case (c)};

\end{tikzpicture}
\caption{\textit{An illustration of the three cases of an acyclic component $T$ in $G^{AB}$ with $|A(T)|=3$, and a covering of them (non-dotted) that demonstrates the equality $a(T)=2n_0(T)+n_1(T)+s'_3(T)$.}}\label{fig:lem5.4}
\end{figure}
\end{proof}

\subsection{Proof of Theorem  \ref{thm:main1}}
Let $f: [0,\infty ) \to [0,1]$ be as defined in Section \ref{subsec:scalinglimit}. Then Lemma \ref{lem:thm1:part2} implies that $\mu(G)=f(d)\cdot n+O(n^{0.5})$ with high probability, and Lemma \ref{lem:thm1:part3} gives the first terms of $f(d)$, from which also follows that $f(d)\cdot n=\omega(n^{0.55})$ for $d\leq 0.4\log n$. Hence $\mu(G)=(1+o(1)) \cdot f(d) \cdot n$ with high probability for $20\leq d\leq 0.4\log n$, as we set out to prove. \qed

\section{Proof of Theorem \ref{main3} and Theorem \ref{main4}} \label{sec:proof3}
Theorem \ref{main3} follows directly from Lemma \ref{lemma:mu} and auxiliary results.

\begin{proof}[Proof of Theorem \ref{main3}] First, note that Theorem \ref{main3} in the regime $np\geq \log n +1.1\log\log n$ follows from Theorem \ref{thm:CF}. For the rest of the proof, we consider the regime  $20\leq np\leq \log n +1.1\log\log n$.

Letting $d=np$, by Theorem \ref{thm:ANAS} (for $np = O(1)$) and Theorem \ref{thm:LUC} (for $np \to \infty$), with high probability we already have $\left[ \log\log n , (1-0.04d^3e^{-d})\cdot n \right] \subseteq \mathcal{L}(G)$, so it remains to construct a set $F$ such that $G\cup F$ spans cycles with length shorter than $\log \log n$ or longer than $(1-0.04d^3e^{-d})\cdot n$. Let $s=|S(G)|$. By Lemma \ref{lemma:4coreProperties}\ref{4core:sizeofAandS} we have $s\ge 0.05d^3e^{-d}n$.

Assume that the set $F$ from Lemma \ref{lemma:mu} exists, an event which occurs with high probability, and recall that $|F|=\mu (G)$. In this case we also have $[3,\log\log n]\cup [(1-0.05d^3e^{-d})\cdot n,n] \subseteq \cL(G\cup F)$. Overall $\cL(G\cup F) = [3,n]$, and therefore $\hat{\mu}(G) \le |F| = \mu (G)$. Since $\mu (G) \le \hat{\mu}(G)$ for every graph, this implies Theorem \ref{main3}.
\end{proof}

Theorem \ref{main4} follows similarly, with some additions addressing computational complexity.

\begin{proof}[Proof of Theorem \ref{main4}]
With high probability, the set $F$ from Lemma \ref{lemma:mu} exists. As shown in the proof of Theorem \ref{main3}, if $F$ exists, then, with high probability, it has size $\mu (G)$, and it completes $G$ to pancyclicity. It therefore remains to show that there is a polynomial time algorithm that calculates $F$ with high probability if it exists (and calculates a possibly meaningless set, otherwise).

To prove this we will use Lemma \ref{lem:calculate}, stated below.
\begin{lemma}\label{lem:calculate}
    Let $T$ be an acyclic graph whose vertices are coloured blue or red. Let $f(T)$ be the minimum integer such that there exists a disjoint path cover of $T$ with $f(T)$ red endpoints. Then, $f(T)$, along with a corresponding set of paths, can be computed in polynomial time.
\end{lemma}

\begin{proof}
Given $T$, let $A$ and $B$ be the red and blue vertices of $T$, respectively. 
For $k\geq 0$ we construct the auxiliary graph $T_k$, starting from $T$, by adding two sets of vertices $H=\{h_1,h_2,...,h_k\}$ and $J=\{j_1,j_2,j_3\}$, along with the following edges. We add all the edges spanned by $H\cup J\cup B$ and all the edges from $H$ to $A$. Now observe that if $f(T) \in \{2k-1, 2k\}$ then $T_k$ has a $2$-factor, i.e. a spanning $2$-regular graph. 

Indeed, let $\{P_1,...P_\ell\}$ be a disjoint path cover of $T$ with at most $2k$ endpoints in $As$. Starting from the set $\{P_1,...,P_\ell\}$, iteratively connect a pair a paths, each having an endpoint in $B$, via an edge between an endpoint in $B$ of each path. This is possible since $B$ is a clique in $T_k$. This results in a set of vertex disjoint paths $\{P_1',...P_{\ell'}'\}$ in $T_k$, whose vertex set is $A\cup B$, such that the paths have at most $2k$ endpoints in $A$ in total, and at most one of the paths has endpoints in $B$, say the path $P_{\ell'}$. 

If $P_{\ell'}$ has at most one endpoint in $B$, then by counting the endpoints in $A$ we have that $\ell'=k$. Otherwise, $P_{\ell'}$ has two endpoints in $B$ and $\ell'=k+1$. In both cases let $C$ by the cycle  $h_1P_1'h_2P_2'\cdots h_kP_k'$. If $\ell'=k+1$ then let $C'$ by the cycle $j_1j_2j_3P_{k+1}'$, else let $C'$ be the cycle $j_1j_2j_3$ (recall $B\cup J$ is a clique in $T_k$). In both cases the union of $C$ and $C'$ is a $2$-factor of $T_k$.

On the other hand if $T_k$ has a $2$-factor $P$, then by removing from $E(P)$ all the edges that are either spanned by $B\cup J\cup H$ or are from  $H$ to $A$ we get a subgraph $F$ of $T$ of maximum degree $2$. Since $T$ is acyclic, $F$ corresponds to a disjoint path cover of $T$. The endpoints of paths in $F$ either lie in $B$ or are incident to a removed edge from $H$ to $A$. As there are at most $2|H|=2k$ edges of the second kind we have that $F$ has at most $2k$ endpoints in $A$. 

Let $r(T)$ be the minimum $k$ such that $T_k$ has a $2$-factor. From the above it follows that if $T$ is acyclic, then $\lceil \frac{1}{2} f(T) \rceil=r(T)$, so we have that $f(T) \in \{ 2r(T), 2r(T)-1 \}$. To determine $f(T)$ let $T'$ be the graph obtained from $T$ by adding to it an edge component with a red and a blue endpoint. If $r(T)=r(T')$ then $f(T)=2r(T)-1$, and otherwise $f(T)=2r(T)$. In each of the cases, as seen earlier, a disjoint path cover with $f(T)$ endpoints in $A$ can be derived from a $2$-factor of an auxiliary graph.  Determining whether a graph $H$ has a $2$-factor, and finding one in the case that it has, can be done in polynomial time in the number of vertices of $H$ (see the sole lemma of \cite{tutte}). This gives a polynomial time algorithm that computes $f(T)$, along with a corresponding set of paths.
\end{proof}

In order to now compute $F$, first compute the partition $A(G),B(G),C(G)$. This is done in polynomial time, as there are at most $2n$ steps to the colouring process (see Section \ref{sec:4core}), and in each step we consider each vertex at most once. From $A(G),B(G)$, the connected components of $G^{AB}$ can now be computed in polynomial time. If one such component contains more than $\log \log n$ vertices and a cycle, return $F = \emptyset$. By Lemma \ref{lemma:4coreProperties}\ref{4core:notrees}, this occurs with probability $o(1)$. We continue under the assumption that no such component exists in the input graph.

From here, we follow the process of building $F$ in the proof of Lemma \ref{lemma:mu}. First, by Lemma \ref{lem:calculate}, a minimum size path covering of $T$ can be found in polynomial time for every tree component $T\in \mathcal{T}(G^{AB})$, where the colouring on the vertices of $T$ is determined by colouring the vertices in $A(G)\cap T$ red and the rest of the vertices  blue. For the non-tree component, since we are assuming they all have at most $\log \log n$ vertices, a minimum size covering can be found by examining all the subsets of $E(T)$ in $O(2^{\binom{\log\log n}{2}})$ time.

With this we have a path covering $Q$ of $G^{AB}$ with $\lceil \frac{1}{2}a(Q) \rceil = \mu '(G)$, same as the covering in the proof of Lemma \ref{lemma:mu}. The last remaining part is calculating the sets $F_0,F_1,F_2$, the union of which is $F$, but, following their construction in the proof of Lemma  \ref{lemma:mu}, they are calculated directly (and in at most linear time) from $Q$.
\end{proof}
 
\begin{remark}
The algorithm in the proof of the above lemma finds a disjoint path cover $Q$ of $G^{AB}$ that minimizes $a(Q)$, that is,  $a (Q) = a(G)$, in polynomial time with high probability. This results in the lower bound  $\lceil \frac{1}{2} a(G) \rceil$ of $\mu(G)$.  Recall that the two quantities are equal with high probability. One may go one step further and utilize the arguments in the proofs of Lemma \ref{lemma:4coreHmty} and Lemma \ref{lemma:mu} (the first is based on P\'osa rotations)     
    to derive an algorithm that finds a disjoint path cover $P$ of $G$ of  size $\lceil \frac{1}{2} a(G) \rceil$ with high probability. This algorithm can be used to certify the equality. As the proof of Lemma \ref{lemma:4coreHmty} is not in the scope of this paper we, omit further details.   
\end{remark}

\section{Proof of Theorem \ref{main2}} \label{sec:proof2}
For the purpose of proving Theorem \ref{main2}, we may assume that $g(n)=O(\log\log \log n)$.
\subsection{Theorem \ref{main2} for \texorpdfstring{$t<10n$}{Lg}}

%We may assume that $g(n)=O(\log\log \log n)$.
It is well known that $G_{g(n)^{-1}\cdot n^{2/3}}$ and $G_{g(n)\cdot n^{2/3}}$ consist of tree components with at most $3$, and $4$ respectively, vertices with high probability. In addition, $G_{g(n)\cdot n^{2/3}}$ has $\omega(1)$ stars with $3$-leaves, i.e. $K_{1,3}$'s with high probability (for reference, see for example \cite{FK} \textsection 2.1). Thus the number of $K_{1,3}$ components in $G_t$ is $0$ for $t=g(n)^{-1}\cdot n^{2/3}$, increasing for $g(n)^{-1}\cdot n^{2/3}\leq  t \leq g(n) \cdot n^{2/3}$, and $\omega(1)$ for $t=g(n)\cdot n^{2/3}$ with high probability. 
As $\mu(T)=n_0(T)+ \frac{1}{2}n_1(T) $ for every tree $T$ on one to four vertices other than the $3$-star $K_{1,3}$, and $\mu(K_{1,3})= 2 = n_0(K_{1,3})+ \frac{1}{2}(n_1(K_{1,3}) +1)$, parts \ref{main2:part1} and \ref{main2:part2} of Theorem \ref{main3} follow.
In addition, one can show that, a constant portion of the $K_{1,3}$ components of $G_{g(n)\cdot n^{2/3}}$ are also $K_{1,3}$ components of $G_{t}$ with high probability. Thus the graph $G_t$ spans $\omega(1)$ many $K_{1,3}$-components, for every $g(n)\cdot n^{2/3} \le t \le 10n$ with high probability. As $\mu(K_{1,3})> n_0(K_{1,3})+\frac{1}{2} n_1(K_{1,3})$ this implies \ref{main2:part3} for the range $g(n)\cdot n^{2/3} \le t\leq 10n$.

\subsection{Theorem \ref{main2} for \texorpdfstring{$t\geq 10n$}{Lg}}

For the rest of this section, we utilize the definitions of $A(G),B(G),C(G),S(G)$ introduced in Section \ref{sec:4core}, and for a connected component $T$ in $G^{AB}$ the definitions of $A(T),B(T),n_i(T),a(T),s_3'(T)$ introduced in Section \ref{subsec:approx}. Recall that we denote by $s_3(G)$ the number of $3$-spiders of $G$. We further denote by $s_3(T)$ the number of $3$-spiders of $G$ whose edge set is spanned by $T$.

Observe that if $X$ is a 3-spider in $G$ then $X$ is part of a component $T$ in $G^{AB}$ such that the three degree 2 vertices of $X$ belong to $A(T)$, and any disjoint path covering of $T$ leaves at least one of the degree 2 vertices of $X$ as an endpoint of a path. This means that if $s_3(G) \ge i$ then one has $\mu (G) \ge \mu '(G) \ge n_0(G) + \lceil \frac{1}{2} (n_1(G) +i) \rceil$ since, in this case, in every disjoint path covering of $G^{AB}$ there are at least $i$ degree 2 vertices in $A(G)$ that are endpoints of paths. Denoting by $t_i$ the minimum $t$ such that $t \ge 10n$ and $s_3(G_t) < i$, this already proves part \ref{main2:part3} of the theorem for the range $t\geq 10n$, given that $t_i$ is indeed shown to be larger than $n\cdot \left( \frac{1}{6}\log n + \log \log n -g(n)\right)$ with high probability for every $i\in \mathbb{N}^+$.

In order to prove parts \ref{main2:part4} and \ref{main2:part5} of the theorem, we begin by proving versions of Lemma \ref{lem:4coreProperties2} and Lemma \ref{lemma:mu} in the setting of a random graph process.

\begin{lemma}\label{lem:4coreProperties:process} 
Let $\{ G_t \} _{t=0}^{\binom{n}{2}}$ be a random graph process on $[n]$. Then with high probability the event $\cE_1(G_t)$ occurs for all $\frac{1}{6}n\log n \le t \le \frac{3}{5}n\log n$.
\end{lemma}

A proof of Lemma \ref{lem:4coreProperties:process} is found in Appendix \ref{sec:app:4coreProperties:process}.

\begin{lemma} \label{lemma:mu:process}
Let $\{ G_t \} _{t=0}^{\binom{n}{2}}$ be a random graph process on $[n]$. Then, with high probability, for every $t\ge \frac{1}{6}n\log n$ there exists a set $F_t\subseteq \binom{[n]}{2}$ of size $\mu '(G_t)$ such that $G_t\cup F_t$ is Hamiltonian.
\end{lemma}

\begin{proof}

First, with high probability $G_t$ is Hamiltonian for all $t\ge \frac{3}{5}n\log n$, in which case $F_t=\emptyset$ satisfies the requirement, so it suffices to prove the lemma for $\frac{1}{6}n\log n \le t \le \frac{3}{5}n\log n$. The rest of the proof is identical to the proof of Lemma \ref{lemma:mu} which gives that if for some $\frac{1}{6}n\log n \le t \le \frac{3}{5}n\log n$ there does not exist a set $F_t\subseteq \binom{[n]}{2}$ of size $\mu '(G_t)$ such that $G_t\cup F_t$ is Hamiltonian then the following holds. Either $\cE_1(G_t)$ does not occur or there exist $U\subseteq B(G)$ and a matching $M \subseteq \binom{U}{2}$ on the set $U$ such that $G\left[ C(G) \cup U \right] \cup M$ does not contains a Hamilton cycle that spans all the edges of $M$. Lemma \ref{lemma:4coreHmty} and Lemma \ref{lemma:equiv} imply that the later occurs with probability $O(n^{-2} \cdot \sqrt{t})=o(n^{-1.1})$. Lemma \ref{lem:4coreProperties:process} states that $\cap_{t=\frac{1}{6}n\log n }^{\frac{1}{6}n\log n} \cE_1(G_t)$ occurs with high probability. It follows that there exists a set $F_t\subseteq \binom{[n]}{2}$ of size $\mu '(G_t)$ such that $G_t\cup F_t$ is Hamiltonian for  $\frac{1}{6}n\log n \leq t\leq \frac{3}{5}n\log n$ with probability at least 
$$\pr\left(\bigcap_{t=\frac{1}{6}n\log n }^{\frac{1}{6}n\log n} \cE_1(G_t)\right)-\sum_{t=\frac{1}{6}n\log n }^{\frac{3}{5}n\log n} o(n^{-1.1}) =1-o(1).$$
\end{proof}

Now let $t^-:= n\cdot \left( \frac{1}{6}\log n + \log \log n -g(n)\right)$ and $t^+ = n\cdot \left( \frac{1}{6}\log n + \log \log n +g(n)\right)$. For parts \ref{main2:part4} and \ref{main2:part5}, we show that under certain conditions, which hold with high probability for $G_t$ with $t\ge t^-$, we have $\mu (G_t) = n_0 + \lceil \frac{1}{2} \cdot (n_1 (G_t) + s_3 (G_t)) \rceil$. We then argue that $s_3(G_{t})=0$ for $t\geq t^+$ and $s_3(G_t)$ is non-increasing for $t\in [t^-,t^+]$ with high probability.

We begin by showing that there is a small family of graphs such that, with high probability, all the connected components in ${G_t}^{AB}$ are induced copies of graphs from this family, for all $t^-\le t \le \tau _{\cH}$, where $\tau _{\cH}$ is the hitting time of Hamiltonicity.

Lemmas \ref{lem:algebra} and \ref{lemma:proofcomps},  stated shortly, are used in the proof of Lemma \ref{lemma:bounding-ti}. Their proofs are located in Appendixes \ref{app:lemma:algebra} and \ref{app:lemma:proof3comps} respectively. The proof of Lemma  \ref{lemma:proofcomps}
is based on first moment calculations.

\begin{lemma} \label{lem:algebra}
Let $x,y,z$ be non-negative integers and let $t=t(n)=\Theta(n\log n)$. 
Let $N=\binom{n}{2}$. Then,
\begin{equation}\label{eq:algebra}
        \frac{\binom{N-(xn-y)}{t-z}}{\binom{N}{t}}= (1+o(1)) \bfrac{t}{N}^z \cdot \exp\left( -\frac{xnt}{N}\right).    
\end{equation}
\end{lemma}

\begin{lemma} \label{lemma:proofcomps}
With high probability, for every $t^- \le t \le \tau_{\mathcal{H}}$, every connected component of ${G_t}^{AB}$ is a tree with at most two vertices from $A(G_t)$, or a tree with three vertices from $A(G_t)$ and at least four vertices from $B(G_t)$.
\end{lemma}

Lemma \ref{lemma:proofcomps} implies that with high probability, for every  $t^-\leq t \leq t_{\cH}$, every $3$-prespider of $G_t$ is in fact a $3$-spider. Hence Lemma \ref{lemma:proof3comb} gives that, with high probability,
\begin{equation}\label{eq:mutlarge}
    \mu'(G_t)=  n_0(G) +\bigg \lceil  \frac{1}{2} \cdot \left( n_1(G) + s_3(G) \right) \bigg \rceil \text{ for every $t^-\leq t \leq t_{\cH}$.}
\end{equation}

Finally, we bound $t_i$ and show that, while $s_3(G)<i$ is not a monotone property, typically $s_3(G_t)<i$ for all $t\ge t_i$, $i \in \mathbb{N}^+$.

\begin{lemma} \label{lemma:bounding-ti}
Let $g(n)$ be any function that tends to infinity as $n$ tends to infinity. Then with high probability, for $i \in \mathbb{N}^+$, 
$$
\left| t_i-n\cdot \left( \frac{1}{6}\log n +\log\log n \right) \right| \le n\cdot g(n) ,
$$
and $s_3(G_t)<i$ for every $t\ge t_i$. In addition, $s_3(G_t)\geq s_3(G_{t+1})$ for $t\geq t^-$ with high probability.
\end{lemma}

\begin{proof}
Let $N=\binom{n}{2}$
First, we show that if $t = t^+$ then with high probability $s_3(G_t)=0$, and hence $t_i\leq t^+$ for $i\in \mathbb{N}^+$.  We do this by the union bound. To upper bound $\mathbb{E}(s_3(G_t))$ observe that each, of the at most $n^7$, $3$-spiders $s$ belongs to $G_t$ only if  $E(G_t)$ contains the $6$ edges spanned by $s$ and the other $t-6$ edges of $G_t$ do not include one of the $3(n-3)+\binom{3}{2}=3n-6$ additional edges that are incident to the vertices of degree $2$ of $s$. There exist $\binom{N-(3n-6)}{t-6}$ such sets of edges out of the $\binom{N}{t}$ ones that have size $t$. Thus, using \eqref{eq:algebra} we get,
\begin{eqnarray*}
\pr (s_3(G_t) >0) \le \mathbb{E}(s_3(G_t)) & \le & n^7 \cdot \frac{\binom{N-(3n-6)}{t-6} }{\binom{N}{t}} 
=n^7 \cdot (1+o(1)) \cdot\bfrac{t}{N}^6 \cdot \exp\left(-\frac{3tn}{N}\right) \\
\\
& = & O(1) \cdot n\cdot \log ^6n \cdot \exp \left( -\log n - 6\log \log n - 6g(n)\right) \\
& = & O(1)\cdot e^{-6g(n)} = o(1).
\end{eqnarray*}

Next, we show that, with high probability, $s_3(G_t)\ge i$ for all $i\in \mathbb{N}^+$ and $10n \le t \le t^-$, which implies that $t_i \ge t^-$ for $i\in \mathbb{N}^+$. Let $t' \coloneqq n\cdot \left( \frac{1}{18}\log n + \log \log n \right)$. First, observe that Lemma \ref{lemma:equiv} and Lemma \ref{lemma:4coreProperties} \ref{4core:3spiders} imply that 
$s_3(G_t) \geq 10^{-8}n^{2/3}$ for $10n\leq t \leq t'$ with probability $1-t'\cdot o(n^{-2})=1-o(1).$ To show that $s_3(G_t)>i$ for $t'\leq t\leq t^-$ we argue, as $G_{t'}$ typically spans $\Omega \left( n^{2/3} \right)$ 3-spiders, that between times $t'$ and $t^-$ it is unlikely that too many of the 3-spiders spanned at time $t'$ disappear.

For that, fix a set $S_3$ of $\min\{s_3(G_{t'}),10^{-8}n^{2/3}\}$ distinct  $3$-spiders in $G_{t'}$ and define the random variable $X$ to be the number of 3-spiders $s\in S_3$ such that none of the edges in $E(G_{t^-})\setminus E(G_{t'})$ are adjacent to one of the three degree $2$ vertices of $s$. Note that in the event that ${G_{t^-}}^{AB}$ does not span a component that intersects $A(G_{t^-})$ in at least $4$ vertices we have that $X \le s_3(G_{t^-})$; Lemma \ref{lemma:4coreProperties}\ref{4core:sizeX_i} and Markov's inequality imply that this occurs with high probability. We bound the probability that $X < i$ by estimating its expectation and variance, and invoking Chebyshev's inequality. 

There are $\binom{N-t'}{t^--t'}$  sets of edges of size $t^--t$ that do not intersect $E(G_{t'})$, each corresponding to a possible realisation of $E(G_{t^-})\setminus E(G_{t'})$. Out of those, for a fixed $s\in S_3$, exactly $\binom{N-t'-(3n-6)}{t^--t'}$ do not contain an edge out of the $3n-6$ non-edges in $G_{t'}$ that are adjacent to one of the three degree $2$ vertices of $s$. Thus, \eqref{eq:algebra} gives that, 
\begin{align*}
    \mathbb{E}(X)& = |S_3| \frac{\binom{N-t'-(3n-6)}{t^--t'} }{ \binom{N-t'}{t^--t'} } 
= |S_3| \cdot (1+o(1)) \exp\left(- \frac{3(t^--t')n}{N} \right) 
\\&= (1+o(1)) |S_3|  \exp\left(- \frac{2}{3}\log n +6g(n) \right) = (1+o(1)) |S_3| n^{-2/3} e^{6 g(n)}.
\end{align*}

Similarly, for distinct $s,s'\in S_3$,  exactly $\binom{N-t'-(6n-21)}{t^--t'}$ do not contain an edge out of the $6(n-6)+\binom{6}{2}=6n-21$ non-edges in $G_{t'}$ that are adjacent to one of the, six in total, degree $2$ vertices of $s$ or $s'$. Thus, \eqref{eq:algebra} gives that, 
\begin{eqnarray*}
\mathbb{E}(X^2)&=& \mathbb{E}(X)+ \mathbb{E}(X(X-1))= \mathbb{E}(X) +|S_3|(|S_3|-1) \frac{\binom{N-t'-(6n-21)}{t^--t'}}{\binom{N-t'}{t^--t'}}\\& =& \mathbb{E}(X) +|S_3|(|S_3|-1) (1+o(1)) \exp\left(- \frac{6(t^--t')n}{N} \right)
\\&=&    \mathbb{E}(X) + (1+o(1)) (\mathbb{E}(X))^2=(1+o(1)) (\mathbb{E}(X))^2.
\end{eqnarray*}
Recall that Lemma \ref{lemma:equiv} and Lemma \ref{lemma:4coreProperties} \ref{4core:3spiders} imply that $s_3(G_{t'}) \geq 10^{-8}n^{2/3}$, hence $|S_3|\geq 10^{-8}n^{2/3}$ with high probability. In this event, Chebyshev's inequality implies that $X \ge \frac{1}{2}|S_3| \cdot n^{-2/3}e^{g(n)}=\Omega(e^{6g(n)})$ with high probability. Overall we have that, with high probability, for every $t'\le t\le t^-$ and $i\in \mathbb{N}^+$ the inequality $s_3(G_t) \ge i$ holds.

Finally, we show that with high probability, there is no $t \ge t^-$ such that $s_3(G_t)>s_3(G_{t-1})$. New 3-spiders cannot be created if $G_{t-1}$ does not contain vertices of degree 1, so it is enough to check this up to the hitting time of Hamiltonicity, which is with high probability at most $n(\log n+1.1\log\log n)$.

Observe that, on the high probability event in the assertion of Lemma \ref{lemma:proofcomps}, if the $t$-th step of the process introduces a 3-spider in $G_t$ that did not exist in $G_{t-1}$, then the edge added connects a degree 1 vertex to the middle vertex of a path $P=v_1v_2v_3v_4v_5$ where $v_2$ and $v_4$ have degree $2$ in $G_{t-1}$. Similarly to the calculation at the beginning of the proof of this lemma, the probability that $G_t$ contains a $3$-spider and one of its edges appears at the $t$-th step of the process (as the order in which the edges appear is uniform, conditioned on $E(G_t)$, each of the edges in $E(G_t)$ is the last one with probability $1/t$) is at most 
$$ n^7 \cdot \frac{\binom{N-(3n-6)}{t-6} }{\binom{N}{t}} \cdot \frac{3}{t} =(1+o(1)) \cdot n^7 \cdot \bfrac{t}{N}^6  \cdot \exp \left( -\frac{3tn}{N} \right)\cdot  \frac{3}{t}=O(\log^5 n) \cdot \exp \left( -\frac{6t}{n} \right).$$
Thus, the probability that there exist $t\in [t^-,n(\log n+1.1\log\log n]$ such that $G_t$ has a $3$-spider that is no present in $G_{t-1}$ is at most, 
\begin{eqnarray*}
& & o(1)+ O(\log ^5n )   \sum _{t=t^-}^{n(\log n+1.1\log\log n)} \exp \left(\frac{6t}{n}\right) = o(1)+  O(\log ^5n ) \cdot \frac{\exp\left(-\frac{-6t^-}{n}\right)}{1-e^{-6/n}}
\\&\leq &  o(1)+ O(\log ^5n ) \cdot \frac{\exp (-\log n -(6+o(1)) \log\log n)}{1/n} = o(1)+O(\log^{-1+o(1)}n)=o(1),
\end{eqnarray*}
where at the last inequality we use that bound $1-e^{-6/n} \ge \frac{1}{n}$, which holds for $n\ge 1$.
\end{proof}
To finish the proof of Theorem \ref{main2} observe that Observation \ref{obs:mu}, Lemma \ref{lemma:mu:process} and \eqref{eq:mutlarge} imply that, with high probability, for all $(n\log n)/6\leq t\leq (3n/\log n)/5$ we have that,
\begin{equation*}
    \mu(G_t)= \mu'(G_t)=  n_0(G) +\bigg \lceil  \frac{1}{2} \cdot \left( n_1(G) + s_3(G) \right) \bigg \rceil.
\end{equation*}
Parts \ref{main2:part4} and \ref{main2:part5}  of Theorem \ref{main2} follow from the above equation and Lemma \ref{lemma:bounding-ti}.

\section{Concluding remarks}

The definition of the completion number extends naturally to the setting of digraphs. In this setting as well, the completion number of a digraph $D$ with respect to Hamiltonicity is equal to the minimum number of vertex disjoint (directed) paths required to cover $V(D)$. It would be interesting to see if similar tools to those we used here can be used to prove an equivalent of Theorem \ref{thm:main1} in $D(n,p)$ (that is, to show that $\mu (D(n,p)) = (1+o(1))\cdot  \overrightarrow{f}(np)\cdot n$ with high probability, for an appropriate function $ \overrightarrow{f}$, in some range of $p$). In a recent paper, Anastos and Frieze \cite{AFDIR} studied a substructure of digraphs similar in idea to the strong $k$-core. This, in particular, may prove useful for studying $\mu (D(n,p))$. It should be noted that, in such an attempt, one should be careful in relegating the problem of bounding the number of paths in a minimum cover of $D$ to the problem of covering connected components outside of the (equivalent of the) strong core, since, in the directed case, the direction of the paths (``into the core" or ``away from the core") now matters.

The result in \cite{AFDIR}, combined with a recent result by Alon, Krivelevich and Lubetzky \cite{AKL}, show that, similar to the undirected case (see Theorem \ref{thm:ANAS}), if $c$ is a large enough constant then $\cL (D(n,c/n))$ with high probability contains all ``not too short" and ``not too long" cycle lengths. So, in adding arcs to create a pancylic digraph, only a few cycle lengths remain to be added. It would be interesting to see if $\hat{\mu}(D(n,p))$ is also equal to $\mu(D(n,p))$ with high probability.\\

\noindent \textbf{Acknowledgements.} The authors would like to express their thanks to the referees of the paper for their valuable input towards improving the presentation of our result.

\begin{appendix}
\section{Proof of Lemmas \ref{lemma:4coreProperties} and \ref{lem:muconcentration}}\label{app:sec:lemma:4coreProperties} 
\begin{proof}[Proof of Lemma \ref{lemma:4coreProperties}]
We start by proving \ref{4core:sizeX_i}. Let $T$ be a component of $G^{AB}$ that intersects $A(G)$ in $i$ vertices and $B(G)$ in $j$ vertices. During the colouring procedure that identifies the strong $4$-core of $G$, described earlier in Section \ref{sec:4core}, every time a vertex is coloured red at most $3$ of its neighbours are coloured blue. Thus, $j\leq 3i$. Additionally,  $G$ contains no edges between the $i$ vertices in $T\cap A(G)$ and the at least $n-4i$ vertices outside $T$. Finally, $T$ has a spanning tree. The expected value of $X_i$, where $i\le \log ^2n$, is therefore at most
 
\begin{eqnarray*}
\mathbb{E}(X_i) & \le & i\sum_{j=0}^{3i}\binom{n}{i+j}\cdot\binom{i+j}{j} \cdot (i+j)^{i+j-2}\cdot p^{i+j-1} \cdot (1-p)^{i\cdot (n-4i)} \\
& \le & i\sum_{j=0}^{3i}\bfrac{en}{i+j}^{i+j}\cdot 2^{i+j}\cdot (i+j)^{i+j-2}\cdot \bfrac{d}{n}^{i+j-1} \cdot 1.01e^{-inp} 
\\& = & \frac{1.01n e^{-id}}{d}\cdot \sum_{j=0}^{3i} \frac{i}{(i+j)^2} \cdot  (2ed)^{i+j} \leq  \frac{1.01n e^{-id}}{d}\cdot 1.05  \cdot \frac{i}{(i+3i)^2 } \cdot (2ed)^{i+3i} \leq  \frac{(2ed)^{4i} \cdot e^{-id}}{15id}\cdot n.
\end{eqnarray*}

For the case $i> \log ^2n$ consider the construction of $A(G)$ by the colouring procedure. For $j>0$, let $T_j$ be the subgraph of $T$ spanned by the vertices of $T$ that have colour red or blue after the first $j$ steps of the colouring procedure. %At each step of the procedure, consider the connected components of the subgraph spanned only by the vertices of $T$ that are currently coloured red or blue.
Observe that at Step $j$, a vertex is coloured red and its (at most three) black neighbours are coloured blue, thus at most $4$ vertices are coloured in total. The edges that do not belong to $T_{j-1}$ but belong to $T_j$ are incident to those, at most $4$, vertices. Therefore, $T_j$ has at most $4$ vertices and $4\Delta(G)$ edges more than $T_{j-1}$. It follows that if $T_{j-1}$ is non-empty, then the maximum number of vertices spanned by a single component of $T_j$ is at most $4\Delta(G)+1$ times larger than the maximum number of vertices spanned by a single component of $T_{j-1}$. On the other hand, if $T_{j-1}$ is empty, then $T_j$ has at most $4$ vertices.
%So by colouring a new vertex $v$ red, and its at most three black neighbours blue, the connected component in $T$ joins the set of the at most $4$ vertices just recoloured, say $R$, and the at most $4\Delta (G)$ connected components of the red-blue part of $T$ prior to the colouring of $v$ that are incident to $R$. Also observe that, at any stage of the procedure, there are no edges between red vertices and black vertices.
Therefore, if $T$ is a connected component of size at least $i> \log ^2n$, then there exists $j$ such that $T_j$ has a component $T'$ of size between $\log ^2n$ and, say, $(4\Delta (G)+1)\log ^2 n\leq 5\Delta (G)\cdot \log ^2 n$. Now, due construction, the vertices that have colour red after $j$ steps of the procedure in $T'$ have no neighbours outside $T'$. %, at least a quarter of whose vertices have no edges outside of the subgraph. 
By the same arguments as in the case $i<
\log^2 n$ (with $T'$ in place of $T$, and considering only the first $j$ steps of the colouring procedure), we have
\begin{eqnarray}  \label{eqn:large-components}
\mathbb{E}(X_i) \le n\cdot \pr \left( \Delta (G) > \log ^2 n \right) +  \sum _{j=\log ^2n}^{5 \log ^4n} \frac{(2ed)^{4j}e^{-jd}}{15jd} \cdot n = o(n^{-6}).
\end{eqnarray}

For part \ref{4core:notrees}, similar to the calculations above, the probability that $G^{AB}$ has a component with more than $\log^2 n$ vertices, or a component that span a cycle with at least $\log\log n$ and at most $\log^2 n$ vertices is at most
\begin{align*}
 &   o(1)+ \sum_{i=0.5\log\log n}^{\log^2 n}\sum_{j=0}^{3i}\binom{n}{i+j}\cdot\binom{i+j}{j} \cdot (i+j)^{i+j-2}\cdot p^{i+j-1} \cdot (1-p)^{i\cdot (n-4i)} \cdot \binom{i+j}{2}p 
  \\&  \leq o(1)+ \sum_{i=0.5\log\log n}^{\log^2 n} \frac{(2ed)^{4i} \cdot e^{-id}}{15i^2d}\cdot n \cdot \binom{4i}{2}\cdot \frac{d}{n}
  \leq o(1)+ \sum_{i=0.5\log\log n}^{\log^2 n}   \left((2ed)^4e^{-d}\right)^i=o(1).
\end{align*}

For parts \ref{4core:sizeofAandS}--\ref{4core:3spiders}, we say that a graph $G$ has the property $\cP_n$ if $G^{AB}$ does not span a component of size larger than $\log^2 n$ and its maximum degree is at most $\log^2 n$. From \eqref{eqn:large-components} it follows that in our range of $p$ we have that $\Pr(G(n,p)\notin \cP_n) =o(n^{-4})+n\Pr(Bin(n,p)\geq \log^2 n)=o(n^{-4}).$
We will use the property $\cP_n$ in combination with the following observation. For $v\in [n]$ and a pair of graphs $G_1,G_2 \in \cP_n$ on $[n]$ that differ only on the edges incident to $v$ and $i=1,2$ denote by $D(G_i,v)$ the set of vertices that lie in some component of ${G_i}^{AB}$  that contain $v$ or a neighbour of $v$ in ${G_i}^{AB}$ (set $D(G_i,v)=\emptyset$ if no such component exists). Set $D= D(G_1,v)\cup D(G_2,v)$. Then, 
\begin{equation}\label{eq:corestability}
    C(G_1)\setminus D= C(G_2)\setminus D, \hspace{5mm}
    B(G_1)\setminus D= B(G_2)\setminus D, \hspace{5mm}\text{and} \hspace{5mm}
    A(G_1)\setminus D= A(G_2)\setminus D.
\end{equation}
Indeed, let $W(v,G_1,G_2)=N_{G_1}(v)\cup N_{G_2}(v)\cup\{v\}$. For each $G_1,G_2$, we may run the colouring process for identifying the strong $4$-core, always progressing a vertex not in $W(v,G_1,G_2)$, if such a vertex exists. We may run the two colouring processes in parallel, so that they are identical until a vertex in $W(v,G_1,G_2)$ is coloured red in one of the two processes. Just before that moment, $v$ has colour black in both colourings, and in extension, the colourings of the two graphs are indeed identical. Moreover, every vertex not in $W(v,G_1,G_2)$ that is not red has at least $4$ black neighbours. Hence, in both of the processes, every vertex that is recoloured after that moment belongs to the same component as some vertex in $W(v,G_1,G_2)$. Equivalently, if a vertex does not belong to $D$, then it receives the same colour from both processes.

$|D|$ is bounded by the size of the largest component in $G^{AB}_1$ or $G^{AB}_2$ times 
the size of $W(v,G_1,G_2)$, which is at most $d_{G_1}(v)+d_{G_2}(v)+1$. As $G_1,G_2\in \cP_n$ we have that  
\begin{equation}\label{eq:corestability2}
|D|\leq \log^2n(\log^2 n+\log^2 n+1)\leq 5\log^4 n.
\end{equation}

We are now ready to prove \ref{4core:sizeofAandS}. Observe that every vertex of degree $3$ that does not lie in a component of $G^{AB}$ with at least $2$ vertices in $A(G)$, determines a unique component in $S(G)$. Denote by $X_{\ge i} \coloneqq \sum _{j\ge i}X_j$. From our calculation in \ref{4core:sizeX_i} we have
\begin{align*}
    \mathbb{E}(|S(G)|) & \geq n\binom{n-1}{3}p^3(1-p)^{n-4}-\mathbb{E}(X_{\ge 2})
   \geq 0.16 d^3e^{-d} n - \sum_{i\geq 2} \frac{(2ed)^{4i}e^{-id}}{15id} 
   \\&\geq d^3e^{-d} n \left(0.16-  \sum_{i\geq 1} \frac{(2ed)^{4i}e^{-id}}{15i}  \right) \geq d^3e^{-d} n \left(0.16-  \sum_{i\geq 1} \frac{0.3^i}{15i}  \right) \geq 0.12 d^3e^{-d} n.
\end{align*} 

Thereafter let $G_1,G_2$ be as above, and observe that \eqref{eq:corestability} implies that every component that belongs to $S(G_1)\triangle S(G_2)$ intersects $D$. \eqref{eq:corestability2} states that $|D|\leq 5\log^4 n$. Hence, Lemma \ref{lemma:martingales} implies
$$
\pr(|S(G)|<\mathbb{E}(|S(G)|)-n^{0.51})\leq 2\exp\left( -\frac{\Omega((n^{0.51})^2)}{n(5\log^4 n+1)^2} \right)+n\cdot n\cdot \pr(G\notin \cP_n)=o(n^{-2}).
$$

Note that $d^3e^{-d} n =\Omega(n^{0.55})$ for $d\leq 0.45\log n$. Thus for $20\leq np\leq 0.45\log n$ we have that $|S(G)|\geq 0.12 d^3e^{-d} n -n^{0.51}\geq 0.1 d^3e^{-d} n$ with high probability.

Now, for $0.45\log n \leq np \leq \log n +1.1\log\log n$ we use Chebyshev's inequality. Simple calculations imply that in this range of $p$ we have that $\mathbb{E}(X_{\geq 2})=o(\mathbb{E}(n_3(G)))$, thus
$$(1+o(1))\mathbb{E}(n_3(G))=
\mathbb{E}(n_3(G)-X_{\geq 2})\leq \mathbb{E}(|S(G)|) \leq \mathbb{E}(n_3(G)),$$
$\mathbb{E}(n_3(G)^2) = (1+o(1))\cdot (\mathbb{E}(n_3(G)))^2$ and $\mathbb{E}(n_3(G)^2) = (1+o(1))\cdot \mathbb{E}(|S(G)|^2)$. Combining these inequalities gives
\begin{align*}
\pr (|S(G)| < 0.1d^3e^{-d} n)
&\leq \pr (|\mathbb{E}(|S(G)|)-|S(G)||\geq 0.1 \mathbb{E}(|S(G)|))
\\&= O\left( \frac{\text{Var}(|S(G)|)}{\mathbb{E}(|S(G)|)^2} \right)
=  \frac{o(\mathbb{E}(n_3(G))^2)}{\mathbb{E}(n_3(G))^2}  = o(1).
\end{align*}

Next we prove \ref{4core:eventE}. First observe that if $np \to \infty$ then by Theorem \ref{thm:LUC} already $[3,\log \log n] \subseteq \mathcal{L} (G)$ with high probability. We may therefore assume that $np = \Theta (1)$.

For $k\in [3,\log \log n]$ let $Y_k$ be the number of tree components of $G^{AB}$
on $k+1$ vertices $b,a_1,...,a_k$, with $b\in B(G)$ and $\{a_1,a_2,a_3,...,a_k\}\subseteq A(G)$, where $b$ is adjacent to $a_1,a_2$ and $a_3$ and $a_3a_4...a_k$ is a path. Similarly to the calculation in \ref{4core:sizeX_i}, given vertices $b',a_1',...,a_k'$  and conditioned on the event that (i) $b'$ is adjacent to $a_1',a_2'$ and $a_3'$, (ii) $a_3'a_4'...a_k'$ is a path, and (iii) the vertices $a_1',...,a_k'$ are not incident to any additional edges, the probability that there exists a components of $G^{ab}$ that spans $b',a_1',...,a_k'$, and $i+j-1$ additional vertices, $j\geq 0$, and it has $i\geq 1$ vertices in $A(G)\setminus \{a_1,...,a_k\}$ is at most
\begin{align*}
& \pr (G\notin \cP_n) + \sum_{i\geq 1} \sum_{j=0}^{3i}\binom{n-k-1}{i+j-1}\cdot\binom{i+j-1}{j-1} \cdot (i+j)^{i+j-2}\cdot p^{i+j-1} \cdot (1-p)^{i\cdot (n-4i)}
\\ &\leq o(1)+ (1+o(1))  \sum_{i\geq 1} \sum_{j=0}^{3i} \frac{1}{(i+j)} \bfrac{i+j}{i+j-1}^{i+j-1} (2ed)^{i+j-1} e^{-id}
\\ &\leq o(1)+ (1+o(1))  \sum_{i\geq 1}  2\cdot \frac{1}{4i} \cdot e \cdot (2ed)^{4i-1} e^{-id}\leq o(1)+ (1+o(1))  \sum_{i\geq 1}  2\cdot \frac{1}{4i}  \cdot e  \cdot \frac{0.3^i}{2ed}
\end{align*}
which is at most $\frac{1}{2}$. We therefore have
$$
\mathbb{E}(Y_k) \ge \binom{n}{k+1}\cdot p^k\cdot (1-p)^{k(n-k-1)+\binom{k+1}{2}-k}\cdot \frac{1}{2} = \Omega (n).
$$
For $v\in [n]$ and a pair of graphs $G_1,G_2 \in \cP_n$ on $[n]$, every component that is not present in both $G_1^{AB}$ and $G_2^{AB}$ intersects $D$, by \eqref{eq:corestability}. Thus,  \eqref{eq:corestability2} gives that $|Y_k(G_1)-Y_k(G_2)|\leq |D|\leq 5 \log^ 4n$. Hence, by Lemma \ref{lemma:martingales} we have that $Y_k \geq 1$ for $k\in [3,\log\log n]$, and therefore the event $\cE$ occurs, with probability at least
$$1-\left( 2\exp\left( -\frac{\Omega(n^2)}{n(5\log^4 n+1)^2} \right)+n\cdot n\cdot \pr(G\notin \cP_n)\right)=1-o(n^{-2}).$$

Finally, we prove \ref{4core:3spiders}. Let $\frac{20}{n}\leq p \leq p'$. Then,
$$
\mathbb{E}(s_3(G))\geq  n\cdot \binom{n}{3}^3  {p}^6 (1-p)^{3(n-7)} \ge \frac{1}{36}\cdot n^7 (p')^6 e^{-3n\cdot p'} \ge \frac{1}{2}\cdot 10^{-7}\cdot n^{2/3}.
$$
Thereafter, observe that if $G_1,G_2\in \cP_n$ differ only in edges adjacent to some vertex $v$, then a $3$-spider belongs only to one of $G_1,G_2$ only if it intersects $D$. Therefore,
$|s_3(G_1) - s_3(G_2)| \leq |D|\leq 5\log^4 n$, by \eqref{eq:corestability2}. Applying Lemma \ref{lemma:martingales} once again gives,
$$
\pr \left( s_3(G)\leq \frac{1}{2}\mathbb{E}(s_3(G)) \right) \leq 2 \exp\left( -\frac{\Omega(n^{4/3})}{2n(5\log^4 n+1)}\right) + n\cdot n\cdot \Pr(G\notin \cP_n) =o(n^{-2}).
$$
Thus $s_3(G)\geq 10^{-8}\cdot n^{2/3}$ with probability $1-o(n^{-2})$.
\end{proof}

\begin{proof}[Proof of Lemma \ref{lem:muconcentration}]
Let $\cP_n$ be as in the proof of Lemma \ref{lemma:4coreProperties} \ref{4core:sizeofAandS}. Then, as in the proof of Lemma \ref{lemma:4coreProperties} \ref{4core:sizeofAandS}, if we let $G_1,G_2\in \cP_n$ be such that the intersection of all the edges in $E(G_1)\triangle E(G_2)$ is incident to some vertex $v$, then the components that are not present in both $G_1^{AB}$ and $G_2^{AB}$ spanned at most $5\log^4 n$ vertices (see \eqref{eq:corestability} and \eqref{eq:corestability2}). Hence, $|\mu'(G_1)-\mu'(G_2)|\leq 5\log^4 n$. Lemma \ref{lemma:martingales} implies
$$
\pr(|\mathbb{E}(\mu'(G))-\mu'(G)|\geq n^{0.51})=o(1).
$$

\end{proof}

\section{Proof of Lemma \ref{lem:4coreProperties2}}\label{sec:app:4coreProperties2} 

\begin{proof}
First observe that if $np\geq \log n+0.1\log\log n$ then with high probability $n_0(G)=0$ and, by Lemma \ref{lemma:4coreProperties}\ref{4core:sizeX_i}, every component of $G^{AB}$ contains exactly $1$ vertex in $A(G)$, and therefore in this range $G^{AB}$ contains no component with two vertices in $A(G)$ and no isolated vertices.

For the range $20 \le np \le \log n+0.1\log\log n$ we follows similar lines to the proof of Lemma \ref{lemma:4coreProperties}\ref{4core:sizeofAandS} and \ref{4core:eventE}
 and show that, with high probability, there are at least two components comprised of a single vertex in $A(G)$ and a single vertex in $B(G)$. Let $Y$ be the number of such components.
As in the proof of Lemma \ref{lemma:4coreProperties}\ref{4core:eventE} the following holds.  For fixed vertices $a,b$, conditioned that $a$ has a single neighbour in $G$, which is $b$, the probability that there exists a components of $G^{ab}$ that spans $1+i+j$ vertices, $j\geq 0$, and it has $i\geq 1$ vertices in $A(G)\setminus \{a\}$ is at most $\frac{1}{2}$. Thus, 
$$
\mathbb{E}(Y)\geq n\cdot (n-1) \cdot p(1-p)^{n-2}\cdot \frac{1}{2}\geq \frac{de^{-d}n}{3}.
$$
If $d\le 0.45\log n$ then this is at least $n^{0.55}$, and we can apply Lemma \ref{lemma:martingales} as in \ref{lemma:4coreProperties}\ref{4core:eventE}. Otherwise, we apply Chebyshev's inequality in the same manner as in Lemma \ref{lemma:4coreProperties}\ref{4core:sizeofAandS}, by utilizing the fact that in this range $\mathbb{E}(X_{\ge 2}) \ll de^{-d}n$.

\end{proof}

\section{Proof of Lemma \ref{lem:4coreProperties:process}}\label{sec:app:4coreProperties:process}  
\begin{proof}
Let $\sigma=\frac{1}{2}n(\log n+\frac{1}{2}\log\log n)$. Let $\cE$ be the event that $G_\sigma$ is connected,  $G_\sigma$ does not contain a pair of vertices $u,v$ within distance at most $3$ from each other, such that $u$ has degree $1$ and the degree of $v$ does not lie in $[0.9\log n,1.1\log n]$, and $G_\sigma$ spans at least two and at most $\log^2 n$ vertices of degree $1$. One can easily show that $\cE$ occurs with high probability.
In the event $\cE$ we have that $\cE_1(G_t)$ occurs for all $t\geq \sigma$ with high probability. 

Assume that $\cE$ occurs, let $u,v$ be two vertices of $G_\sigma$ of degree $1$ and $C_u,C_v$ the components of ${G_\sigma}^{AB}$ containing $u$ and $v$ respectively. Let $G_\sigma(0.2)$ be the random subgraph of $G_\sigma$ where each edge of $G_\sigma$ is present with probability $0.2$, independently. Observe that one can couple $G_\sigma(0.2)$ and $G_{(n\log n)/6}$ such that  $G_\sigma(0.2) \subseteq G_{(n\log n)/6}$ with high probability. Lemma \ref{lemma:4coreProperties} \ref{4core:sizeX_i} and Lemma \ref{lemma:equiv} imply that with high probability, for $\frac{1}{6}n\log n\leq t\leq \sigma$,
no component of ${G_t}^{AB}$ spans more than $10$ vertices in $A(G_t)$, thus more than $40$ vertices in total, and by extension a vertex of $A(G_t)$ that has degree at least $40$ in $G_t$.
Now $\cE_1(G_t)$ occurs if $C_u,C_v\in S(G_t)$. Therefore, for 
$\frac{1}{6}n\log n\leq t\leq \sigma$, the event $\cE_1(G_t)$ does not occur only if there exists $w\in\{u,v\}$ and a vertex $w'$ such that $w'$ is within distance at most $3$ from $w$ in $G_t$, hence in $G_\sigma$, and $w'$ has degree at most $40$ neighbours in $G_t$, hence in $G_\sigma(0.2)$. In the event $\cE$, there are at most $(\log^2 n)^2$ pairs of degree $1$ vertices $\{u,v\}$, and at most $2(1.1\log n)^3$ vertices within distance at most $3$ from $\{u,v\}$ each having degree at least $0.9\log n$. Therefore the probability that $\cE_1(G_t)$ does not occur for some $\frac{1}{6}n\log n\leq t\leq \sigma$ is bounded above by 
$$\pr(\cE)+ O(\log^5 n)\cdot \pr(\Bin(0.9\log n,0.2)\leq 40)=o(1).$$
\end{proof}

\section{Proof of Lemma \ref{lem:algebra}}\label{app:lemma:algebra} 
\begin{proof} Let $N=\binom{n}{2}$. Then,
    \begin{align*}
       &  \frac{\binom{N-(xn-y)}{t-z}}{\binom{N}{t}}= \frac{\prod_{i=0}^{t-z-1}\frac{N-(xn-y)-i}{(t-z)!}}{\prod_{i=0}^{t-z}\frac{N-i}{t!}} = \prod_{i=0}^{t-z-1}\frac{N-(xn-y)-i}{N-i} \cdot \prod_{i=0}^{z-1}(t-i) \prod_{i=t-z}^{t-1} \cdot \frac{1}{N-i}
      \\&= (1+o(1)) \bfrac{t}{N}^z \cdot \prod_{i=0}^{t-z-1}\left(1-\frac{(xn-y)}{N-i}\right) = (1+o(1)) \bfrac{t}{N}^z \cdot \exp\left( \sum_{i=0}^{t-z-1}-\frac{(xn-y)}{N-i} +O\bfrac{n^2}{N^2}\right)   
      \\&= (1+o(1)) \bfrac{t}{N}^z \cdot \exp\left( -\frac{xnt}{N} +O\bfrac{n^2t}{N^2}\right)=(1+o(1)) \bfrac{t}{N}^z \cdot \exp\left( -\frac{xnt}{N}\right).   
    \end{align*}
\end{proof}

\section{Proof of Lemma \ref{lemma:proofcomps}}\label{app:lemma:proof3comps} 

\begin{proof}
Let $N=\binom{n}{2}$. Let $\mathcal{B}$ be the event that there exist $t\in [t^-,n\log n]$ such that $G_t^{AB}$ contains a component with more that $10$ vertices in $A(G_{t})$. Lemmas \ref{lemma:4coreProperties} and \ref{lemma:equiv} imply that,
$$
\Pr(\mathcal{B}) \leq \sum_{i=10}^{\log^2 n}\sum_{t=t^-}^{n\log n} O(\sqrt{t})  \left((4et/(n-1))e^{-2t/(n-1)}\right)^i\cdot n+o(1)=o(1).
$$

Thereafter, if $T$ is a connected component of $G_t^{AB}$ for some $t\in [t^-,n\log n]$, then with $S=V(T)\cap A(G^t)$ and $R=V(T)\cap B(G^t)$ the following holds, with $s=|S|$ and $r=|R|$. (i) $s\leq 3r$,  (ii) $G_{n\log n}[S\cup R]$ contains a spanning tree $T'$, (iii) there exist times $t_1,....,t_{s+r-1}\in [1,n\log n]$ such that at the $t_j$th edge of the process belongs to $E(T')$, for $j=1,...,s+r-1$ (this occurs with probability at most $\bfrac{(s+r-1)n\log n}{N-n\log n}^{s+r-1}$) and (iv) none of the first $t^-$ edges of $G_{n\log n}$, that are not spanned by $T'$ belong to the $s(n-s-r)$ edges from $S$ to $V\setminus (S\cup R)$. 

For $1\leq s \leq 10$ and  $0\leq r\leq 3s$, let $X_{s,r}$ be the number of pairs of disjoint sets $S,R\subseteq V$ of size $s$ and $r$ respectively that satisfy (i)-(iv).  Then,

\begin{eqnarray*}
\mathbb{E}(X_{s,r})& = & n^{s+r}\cdot (s+r)^{s+r-2} \cdot \bfrac{(s+r-1)n\log n}{N-n\log n}^{s+r-1}\cdot 
\frac{\binom{(N-(s+r-1))-s(n-s-r)}{t^-}}{\binom{N-(s+r-1)}{t^-}}
\\ & = & O(1) \cdot n\cdot (\log n)^{s+r-1} \cdot \exp\left( -\frac{st^-n}{N}\right) 
\\ & \leq  &  O(1) n\cdot (\log n)^{s+r-1} \cdot (4s)^{4s} \cdot  \exp \left( -\frac{1}{3}s\log n - 1.8s\log \log n +o(1) \right) 
\\ & = & O(1)\cdot n^{1-\frac{1}{3}s} \cdot \log ^{r-0.8s-1} n.
\end{eqnarray*}
In the event that there exist  $t\in [t^-, \tau_{\mathcal{H}}]$, such that no every connected component of ${G_t}^{AB}$ is a tree with at most two vertices from $A(G_t)$, or a tree with three vertices from $A(G_t)$ and at least four vertices from $B(G_t)$ one of the following holds. Either (i) $\tau_{\mathcal{H}} >n\log n$ or (ii) the event $\mathcal{B}$ occurs or (iii) $Y_{s,r}\geq 1$ for some $0\leq s\leq 10$ and $0\leq r\leq 3s$ or (iv) $X_{s,r}\geq 1$ for some $s=3$ and $r\leq 3$ or (v)  $X_{s,r}\geq 1$ for some $4\leq s\leq 10$ and $0\leq r\leq 3s$. This occurs with probability at most 
$$
o(1)+ O(1)\cdot \sum _{r=0}^{3} \log ^{r-3.4}n + O(1)\cdot \sum_{s=4}^{10} \sum _{r=0}^{3s} O(1)\cdot n^{1-\frac{1}{3}s} \cdot \log ^{r-0.8s-1} n=o(1).$$

\end{proof}

\end{appendix}

\end{document}